\documentclass[a4paper]{article}
\usepackage{geometry}
\usepackage[utf8]{inputenc}
\usepackage{amsmath,amsfonts,amsthm}
\usepackage{hyperref}
\usepackage{comment}
\usepackage{paralist}
\usepackage[normalem]{ulem}

\title{Geometric and o-minimal Littlewood--Offord problems}
\author{Jacob Fox\thanks{Department of Mathematics, Stanford University, Stanford, CA 94305.
Email: \href{mailto:jacobfox@stanford.edu} {\nolinkurl{jacobfox@stanford.edu}}.
Research supported by a Packard Fellowship and by NSF Award DMS-1855635.}\and Matthew Kwan\thanks{Department of Mathematics, Stanford University, Stanford, CA 94305.
Email: \href{mailto:mattkwan@stanford.edu} {\nolinkurl{mattkwan@stanford.edu}}.
Research supported by NSF Award DMS-1953990.}\and Hunter Spink\thanks{Department of Mathematics, Stanford University, Stanford, CA 94305.
Email: \href{mailto:hspink@stanford.edu} {\nolinkurl{hspink@stanford.edu}}.}}
\date{}

\usepackage{amsthm}
\usepackage[nameinlink,capitalise,noabbrev]{cleveref}
\newtheorem{thm}{Theorem}[section]
\crefname{thm}{Theorem}{Theorems}

\newtheorem{prop}[thm]{Proposition}

\newtheorem{cor}[thm]{Corollary}
\newtheorem{qu}{Question}

\newtheorem{conjecture}[thm]{Conjecture}
\crefname{conjecture}{Conjecture}{Conjectures}
\newtheorem{lem}[thm]{Lemma}
\crefname{lem}{Lemma}{Lemmas}
\newtheorem{defn}[thm]{Definition}
\newtheorem{fact}[thm]{Fact}

\theoremstyle{definition}
\newtheorem{rem}[thm]{Remark}

\usepackage{color}
\usepackage{verbatim}
\usepackage{paralist}

\begin{document}

\global\long\def\RR{\mathbb{R}}%
\global\long\def\QQ{\mathbb{Q}}%
\global\long\def\E{\mathbb{E}}%
\global\long\def\Var{\operatorname{Var}}%
\global\long\def\CC{\mathbb{C}}%
\global\long\def\NN{\mathbb{N}}%
\global\long\def\ZZ{\mathbb{Z}}%
\global\long\def\Bad{\operatorname{Bad}}%
\global\long\def\Ber{\operatorname{Bernoulli}}%
\global\long\def\Inf{\operatorname{Inf}}%
\global\long\def\vol{\operatorname{vol}}%
\global\long\def\conv{\operatorname{conv}}%
\global\long\def\floor#1{\left\lfloor #1\right\rfloor }%
\global\long\def\ceil#1{\left\lceil #1\right\rceil }%

\let\originalleft\left
\let\originalright\right
\renewcommand{\left}{\mathopen{}\mathclose\bgroup\originalleft}
\renewcommand{\right}{\aftergroup\egroup\originalright}

\let\OLDthebibliography\thebibliography
\renewcommand\thebibliography[1]{
  \OLDthebibliography{#1}
  \setlength{\parskip}{0pt}
  \setlength{\itemsep}{3pt plus 0.3ex}
}

\global\long\def\mk#1{\textcolor{red}{\textbf{[MK comments:} #1\textbf{]}}}

\global\long\def\hs#1{\textcolor{red}{\textbf{[HS comments:} #1\textbf{]}}}

\newcommand{\hunter}[1]{\textcolor{red}{#1}}

\maketitle

\begin{abstract}
The classical Erd\H os--Littlewood--Offord theorem says that for nonzero vectors $a_1,\dots,a_n\in \RR^d$, any $x\in \RR^d$, and uniformly random $(\xi_1,\dots,\xi_n)\in\{-1,1\}^n$, we have $\Pr(a_1\xi_1+\dots+a_n\xi_n=x)=O(n^{-1/2})$. In this paper we show that $\Pr(a_1\xi_1+\dots+a_n\xi_n\in S)\le n^{-1/2+o(1)}$
whenever $S$ is definable with respect to an o-minimal structure (for example, this holds when $S$ is any algebraic hypersurface), under the necessary condition that it does not contain a line segment. We also obtain an inverse theorem in this setting.

\end{abstract}

\section{Introduction}

Consider a random variable $X=a_{1}\xi_{1}+\dots+a_{n}\xi_{n}$,
where $a_{1},\dots,a_{n}\in\RR$ are real numbers and $\xi_{1},\dots,\xi_{n}\in \{-1,1\}$
are i.i.d Rademacher random variables (meaning that each $\xi_i$ is independent and identically distributed with
$\Pr\left(\xi_{i}=1\right)=\Pr\left(\xi_{i}=-1\right)=1/2$). Broadly speaking, the classical Littlewood--Offord problem
asks for \emph{anti-concentration} estimates for $X$: what can we say about the maximum probability that $X$
is equal to a single value, or falls in an interval of prescribed
length? We refer the reader to \cite{NV13} for a thorough survey of the Littlewood--Offord problem and its applications, beyond the brief (and selective) history in what follows.

In connection with their work on random polynomials, Littlewood and
Offord~\cite{LO43} proved that if each $\left|a_{i}\right|\ge1$,
then $\max_{x\in \mathbb{R}}\Pr\left(\left|X-z\right|\le1\right)=O\left(\log n/\sqrt{n}\right)$.
Erd\H os~\cite{Erd45} later obtained the optimal $O(1/\sqrt{n})$ bound, in what is now known as the Erd\H os--Littlewood--Offord theorem (to see that the $O(1/\sqrt n)$ bound cannot be improved, consider the case where each $a_i=1$). Note that if one is concerned about the probability of taking a particular value instead of falling in an interval, a simple rescaling argument shows that with only the assumption that each $a_i\ne 0$ we have $\max_{x\in \mathbb{R}}\Pr\left(X=x\right)=O\left(1/\sqrt{n}\right)$.

The Littlewood--Offord problem naturally generalises to higher dimensions. For
any $d\in\NN$, consider $a_{1},\dots,a_{n}\in\RR^{d}$, and again
let $X=a_{1}\xi_{1}+\dots+a_{n}\xi_{n}$, where $\xi_{1},\dots,\xi_{n}$
are i.i.d.~Rademacher random variables. Kleitman \cite{K70} proved that if each $\|a_i\|=\sqrt{a_i\cdot a_i}\ge 1$, then $\max_{x\in \RR^d}\Pr(\|X-z\|\le R)\le C_{R,d}/\sqrt{n}$ for some $C_{R,d}$ depending only on $R,d$, and Frankl and F\"uredi~\cite{FF88} later obtained asymptotically optimal bounds on $C_{R,d}$ (see \cite[Section~2]{NV13} for more about the history of this problem). The point-concentration version of the Littlewood--Offord problem generalises much more easily to higher dimensions: if each $a_i\ne 0$, we can project
to a generic 1-dimensional subspace and apply the Erd\H os--Littlewood--Offord theorem to see that the optimal point concentration bound $\max_{x\in \RR^d}\Pr(X=x)=O(1/\sqrt{n})$ still holds.

In this paper we consider a different generalisation of the Littlewood--Offord problem to $\mathbb{R}^d$.

\begin{qu}\label{qu:main}
Assuming only that each $a_i\ne 0$, what geometric constraints can be imposed on a set $S\subset \RR^d$ which imply non-trivial uniform upper bounds on $\Pr(X\in S)$?
\end{qu}

Since we do not make any assumption about the size of the $a_i$ (other than that they are nonzero), this question is much more about the geometry of $S$ than the spatial anti-concentration of $X$, and therefore has a very different flavour to previous Littlewood-Offord variants. For example, when $S$ is a smooth manifold, any answer to this question would have to be applicable both at a local scale (where $S$ resembles an affine subspace), and at a global scale (where the geometry of $S$ becomes more visible).

It seems that questions of this type have not been systematically studied before. However the case where $S\subset \mathbb{R}^2$ is a parabola was considered by Costello in his study of the \emph{quadratic Littlewood--Offord problem}, where he showed  that $\Pr(X\in S)\le n^{-1/2+o(1)}$ by observing a connection to incidence geometry and applying a weighted version of the Szemer\'edi--Trotter theorem \cite[Lemma~13]{Cos13}.


If $S$ contains a line segment (between $x,y\in \RR^d$, say), then taking $a_1=\frac{x+y}{2}, a_2=\dots=a_n= \frac{x-y}{2n}$ yields $\Pr(X\in S)\ge \frac{1}{2}$, so no non-trivial uniform upper bound on $\Pr(X\in S)$ can be established. Highly ``oscillatory'' behaviour is also a problem for us: for example, if $a_1=\dots=a_n=(1,0)\in \RR^2$, then $X$ always lies on the graph of the function $x\mapsto \sin(2\pi x)$.

A very general property that precludes this kind of oscillatory behaviour is the property of being \emph{definable with respect to an o-minimal structure}. We postpone the formal definition to \cref{subsec:o-minimal}, but for now we note that this property holds for any set that can be expressed via a Boolean combination of equalities and inequalities involving compositions of polynomials, the exponential function, and restrictions of real analytic functions to compact boxes (sets obtainable in this way are said to be \emph{explicitly definable}, see for example the survey of Scanlon \cite{S17}). For example, any affine variety, and more generally any semi-algebraic set, is definable with respect to an o-minimal structure.

If $S$ is definable with respect to an o-minimal structure, and does not contain any line segment, we prove that $\Pr(X\in S)\le n^{-1/2+o(1)}$, nearly matching Erd\H os' $O(1/\sqrt{n})$ point concentration bound for the usual Littlewood--Offord problem.

\begin{thm}\label{thm:baby-forwards}
Let $S\subseteq\RR^d$ be a set which is definable with respect to an o-minimal structure and does not contain any line segment, and let $\alpha>0$. If $n$ is sufficiently large in terms of $\alpha,S$, then the following holds.
Consider nonzero $d$-dimensional vectors $a_{1},\dots,a_{n}\in\RR^{d}$, and write
$X=a_{1}\xi_{1}+\dots+a_{n}\xi_{n}$, where $\xi_{1},\dots,\xi_{n}$
are i.i.d.\ Rademacher random variables. Then $$\Pr(X\in S)\le n^{-1/2+\alpha}.$$
\end{thm}

In fact, $n$ only has to be sufficiently large in terms of $\alpha,d$, and the ``complexity'' of $S$ (a notion for definable sets we will introduce in \cref{def:o-complexity}). Every semi-algebraic set (i.e., every Boolean combination of sets in $\mathbb{R}^d$ defined by polynomial equalities and inequalities) is definable with respect to every o-minimal structure, and in this case being of bounded complexity amounts to a bound on the degrees of the polynomials and the number of regions used in the Boolean combination. So, for example, if $n$ is sufficiently large in terms of $\alpha,d$ we have
$$\Pr(\|X-z\|=R)\le  n^{-1/2+\alpha},$$
because for fixed $d$, the spheres of the form $\{x\in \RR^d:\|x-z\|^2=R^2\}$, for $(R,z)\in \mathbb{R}_{>0}\times \mathbb{R}^d$, have bounded complexity.

We remark that the assumption that each $\xi_i$ has a Rademacher distribution (or even that the $\xi_i$ have the same distribution) is actually not essential. Indeed, using standard conditioning arguments (see \cite[Section~4]{MNV16}) one can deduce from \cref{thm:baby-forwards} a stronger result that holds whenever $X$ is a sum of independent random variables which each have reasonably non-degenerate distributions.

We also remark that another restriction that rules out the possibility of ``highly oscillatory behaviour'' is the restriction that $S$ is a set of points in convex position (for example, the boundary of a strictly convex set). We are able to prove a weaker version of \cref{thm:baby-forwards} in this setting; see \cref{subsec:hypergraphs}.

It seems plausible that the optimal point concentration bound $O(1/\sqrt n)$ given by the Erd\H os--Littlewood--Offord theorem also holds in the setting of \cref{thm:baby-forwards}.
\begin{conjecture}\label{conj:forwards}In \cref{thm:baby-forwards} we have $\Pr(X\in S)\le C_{S}/\sqrt n$, for some $C_S$ depending only on $S$.
\end{conjecture}
We are able to prove \cref{conj:forwards} when $S$ has a certain ``generic intersection property'' (see \Cref{def:generic}), which in particular holds when $S\subset \RR^2$ is a convex or irreducible algebraic plane curve, when $S\subset \CC^2\cong \RR^4$ is an irreducible complex algebraic curve (which is of real dimension $2$), and when $S\subset \RR^3$ is the boundary of a 3-dimensional ball (see \Cref{thm:variousS}).

In addition to high-dimensional extensions, many other variants and generalisations of the Littlewood--Offord problem have been proposed and studied. One especially influential direction is the \emph{inverse theory} of the Littlewood--Offord problem, which studies the relationship between the concentration behaviour of $X$ and the arithmetic structure of $a_{1},\dots,a_{n}$. In particular, inverse Littlewood--Offord theorems were essential tools for some of the landmark results in random
matrix theory (see for example \cite{TV09a,TV09b,Tik20}).

Recall that a \emph{generalised arithmetic progression} of \emph{rank} $q$ is a subset of $\mathbb{R}^d$ of the form
$$Q_{b,v_1,\dots,v_q,m_1,\dots,m_q}:=\{b+r_1v_1+\dots+r_qv_q: r_i\in \{0,\ldots,m_i-1\} \text{ for all }i\},$$
where $b,v_1,\dots,v_q\in \mathbb{R}^d$ and $m_1,\dots,m_q\in \mathbb{N}$. Perhaps the most famous inverse theorem, proved by Tao and Vu~\cite{TV09b} (see also the sharpenings in \cite{NV11,TV10}), says that if any point probability $\Pr(X=x)$ is polynomially large (that is, at least $n^{-C}$ for any $C$), then there must be a very strong arithmetic reason: all but $n^\varepsilon$ of the vectors $a_{1},\dots,a_{n}$ must lie in a rank $q$ generalised arithmetic progression of size at most $n^{B}$ (with $q,B$ bounded in terms of $C,\varepsilon>0$).

We prove an analogous inverse theorem (\cref{thm:baby-inverse}) in the setting of \cref{thm:baby-forwards}: if $\Pr(X\in S)$ is polynomially large, then almost all of the vectors $a_1,\dots,a_n$ lie in a small, low-rank generalised arithmetic progression. This follows from the following results for arbitrary (not necessarily definable) $S\subseteq \mathbb{R}^d$.


\begin{thm}
\label{thm:relative}
Consider any subset $S\subseteq \RR^d$, and let $\xi_1,\xi_2,\ldots$ be i.i.d.\ Rademacher random variables. Suppose that there exists $m\ge 1$ such that either
\begin{enumerate}
\item[(a)] $\Pr(a_1\xi_1+\dots+a_{m+1}\xi_{m+1}\in S)<1/2$ for all nonzero $a_1,\dots,a_{m+1}\in \mathbb{R}^{d}$, or, more generally
    \item[(b)] $\Pr(a_1\xi_1+\dots+a_{m}\xi_{m}\in S')<1$ for all nonzero $a_1,\dots,a_{m}\in \mathbb{R}^{d}$ and all translates $S'$ of $S$.
\end{enumerate}
Then, for all $n$ and nonzero $d$-dimensional vectors $a_1,\dots,a_n\in \mathbb{R}^d$, we have $\Pr(X\in S)\le C_m\rho^{1/(m2^{m-1})}$, where $X=a_1\xi_1+\dots+a_n\xi_n$ and $\rho=\max_{x\in \RR^d}\Pr(X=x)$, and  $C_m$ depends only on $m$.
\end{thm}

\begin{proof}[Proof that assumption (a) implies assumption (b)]
Consider any translate $S'=S-b$ of $S$. If $b=0$, we note that $\Pr(a_1\xi_1+\ldots+a_{m-1}\xi_{m-1}+(a_m/2)\xi_{m}+(a_m/2)\xi_{m+1}\in S)<1/2$ implies that $\Pr(a_1\xi_1+\ldots+a_{m}\xi_{m}\in S')<1$, and if $b\ne 0$ we note that $\Pr(a_1\xi_1+\ldots+a_m\xi_{m}+b\xi_{m+1}\in S)<1/2$ implies that $\Pr(a_1\xi_1+\ldots+a_{m}\xi_{m}\in S')<1$.
\end{proof}
In other words, \Cref{thm:relative} says there is a polynomial relationship between probabilities of the form $\Pr(X\in S)$ and the maximum point concentration probability $\max_{x\in \RR^d}\Pr(X=x)$, whenever we have a nontrivial uniform bound on $\Pr(X\in S)$ for any individual $n$.

\Cref{thm:baby-forwards} and \Cref{thm:relative} together imply that whenever $S\subseteq \RR^d$ is a definable set not containing a line segment, then $\Pr(X\in S)=O(\rho^{c_S})$ for some $c_S$ depending on $S$. One can then combine this statement with the Tao--Vu inverse theorem (or any other inverse theorem for the Littlewood--Offord problem). We remark that the logical structure of the paper does not actually proceed in quite this way: we first directly prove an inequality (\cref{thm:pointvsS}) of the form $\Pr(X\in S)=O(\rho^{c_S})$, where $c_S$ only depends on $S$ via its dimension. We then use this inequality, in addition to the Tao--Vu inverse theorem, to deduce an inverse theorem (\cref{thm:baby-inverse}) in our setting, and then use this inverse theorem in the proof of \Cref{thm:baby-forwards}.


We also remark that the condition in \Cref{thm:relative} that $\Pr(a_1\xi_1+\dots+a_{m}\xi_{m}\in S')<1$ (for all translates $S'$ of $S$, and all nonzero $a_1,\dots,a_m\in \RR^d$) is equivalent to the condition that $S$ does not include a generalised arithmetic progression $Q_{b,v_1,\dots,v_m,2,\dots,2}$ for nonzero $v_1,\dots,v_m\in \RR^d$, which is in turn equivalent to the condition that $S$ does not include a Minkowski sum $A_1+\dots+A_m$ with each $|A_i|=2$. As it turns out, the interaction between Minkowski sums and the geometry of $S$ will play a crucial role in what follows.

In the next few subsections we describe the ideas in the proofs of \cref{thm:baby-forwards} and \cref{thm:relative}, and present a few additional related results.

\subsection{Polynomial anti-concentration and semi-algebraic sets}\label{subsec:polynomial}
Recall that $S\subseteq \mathbb{R}^d$ is a semi-algebraic set if it can be expressed as the set of points satisfying a Boolean combination of equations $f=0$ and inequalities $g>0$ for various polynomials $f,g$. We start by introducing a standard notion of complexity for semi-algebraic sets (this is not quite the same as the o-minimal notion of complexity mentioned previously, but as we will see later it is closely related).

\begin{defn}\label{def:description-complexity}
Say that a semi-algebraic subset $S\subset \mathbb{R}^d$ has \emph{description complexity} at most $r$ if it can be written as a Boolean combination of at most $r$ sets of the form $\{x\in \RR^d:f(x)=0\}$ and $\{x\in \RR^d:f(x)>0\}$, for polynomials $f$ of degree at most $r$.
\end{defn}

We next discuss how to prove \Cref{thm:baby-forwards} for semi-algebraic sets, using some of the ideas that have been developed in the study of polynomial anti-concentration. These ideas will also play a key role in the proof of the full (o-minimal) version of \Cref{thm:baby-forwards}.

The \emph{polynomial Littlewood--Offord problem} is concerned with upper bounds on probabilities of the form $\Pr(p(\xi_1,\dots,\xi_n)=x)$, where $p$ is an $n$-variable polynomial satisfying certain conditions (and, as always, $\xi_1,\dots,\xi_n$ are i.i.d.\ Rademacher random variables). This topic has its roots in random matrix theory (see \cite{CTV06,Ngu12}), but there are also a number of connections between the polynomial Littlewood--Offord problem and the theory of Boolean functions (see \cite{MNV16,RV13}). The strongest and most general result in this area is due to Meka, Nguyen and Vu~\cite{MNV16}, establishing near-optimal anti-concentration estimates in terms of a certain combinatorial parameter of the polynomial $p$ (see also \cite{FKS}). 

As observed by Kane (see the discussion in \cite{MNV16}), the result of Meka, Nguyen and Vu can actually be deduced from estimates related to the \emph{Gotsman--Linial conjecture}. Recall that a \emph{degree-$r$ polynomial threshold function} $f:\{-1,1\}^n\to \{-1,1\}$ is a Boolean function of the form $\operatorname{sgn}(p(x))$ for some polynomial $p:\mathbb{R}^d\to \mathbb{R}$ of degree at most $r$, where $$\operatorname{sgn}(y)=\begin{cases}1 & y>0\\-1 & y\le 0.\end{cases}$$ The $i$-th \emph{influence} $\Inf_i(f)$ of a Boolean function $f$ is the probability that, at a random evaluation point $(\xi_1,\dots,\xi_n)\in \{-1,1\}^n$, changing the value of $\xi_i$ would change the value of $f(\xi_1,\dots,\xi_n)$. The \emph{total influence} or \emph{average sensitivity} of $f$ is $$\operatorname{AS}(f)=\Inf_1(f)+\dots+\Inf_n(f).$$
Then, the Gotsman--Linial conjecture is that\footnote{Gotsman and Linial actually conjectured a stronger bound, which has since been disproved; see \cite{Cha18,KMW}.} for any degree-$r$ polynomial threshold function $f$, we have $\operatorname{AS}(f)=O(r\sqrt n)$.
Kane~\cite{Kan14} proved the weaker result $\operatorname{AS}(f)\le C_r (\log n)^{C_r}\sqrt{n}$ for constant $C_r$ depending only on $r$ (see \Cref{thm:Kane-GL}), from which he was able to deduce a strong bound on the polynomial Littlewood--Offord problem (see the discussion in \cite{MNV16}). We are similarly able to deduce the following quantitative version of \Cref{thm:baby-forwards} in the case of semi-algebraic sets.

\begin{thm}[cf.\ \Cref{thm:baby-forwards}]
\label{thm:Gotsman}
Suppose $S\subset \mathbb{R}^d$ is a $k$-dimensional semi-algebraic set of description complexity $r$ not containing a line segment. Consider nonzero $d$-dimensional vectors $a_{1},\dots,a_{n}\in\RR^{d}$, and write
$X=a_{1}\xi_{1}+\dots+a_{n}\xi_{n}$, where $\xi_{1},\dots,\xi_{n}$
are i.i.d.\ Rademacher random variables. Then there are $C_{d,r}$ depending only on $d,r$, and $C_r$ depending only on $r$, such that $\Pr(X\in S)\le C_{d,r}(\log n)^{C_{r}}/\sqrt n$.
\end{thm}
We remark that a full resolution of the Gotsman--Linial conjecture would remove the $(\log n)^{C_r}$ factor, implying \cref{conj:forwards} for semi-algebraic sets not including a line segment. In fact, there is a natural special case of the Gotsman--Linial conjecture that would already suffice, and does not appear to have previously been studied (see \cref{sec:concluding}).

\subsection{The geometry and additive combinatorics of incidence graphs}\label{subsec:graphs}
To prove \Cref{thm:pointvsS}, and to complete the proof of \Cref{thm:baby-forwards} for general sets definable with respect to an o-minimal structure, we study the geometry of $S\subset\mathbb{R}^d$ via the additive combinatorics of finite subsets of $S$. In this subsection we explain how to encode certain combinatorial information in a weighted graph, and prove special cases of \cref{conj:forwards} by considering forbidden subgraphs of this weighted graph. This will serve as a warm-up to the next few subsections, where we extend these ideas to hypergraphs.

For a point set $\mathcal F\subseteq\mathbb{R}^d$ and a family of geometric objects $\mathcal G$ in $\mathbb{R}^d$, the \emph{bipartite incidence graph} $G(\mathcal F,\mathcal G)$ is the graph with parts $\mathcal{F}$ and $\mathcal{G}$, with an edge between $f\in \mathcal{F}$ and $g\in \mathcal{G}$ whenever $f\in g$. This notion is fundamental in incidence geometry: one important way of proving bounds on the number of incidences between $\mathcal{F}$ and $\mathcal{G}$ is to study the subgraphs that $G(\mathcal F,\mathcal G)$ cannot contain, due to the geometry of $\mathcal{F},\mathcal{G}$, and to then apply a theorem from extremal graph theory. For example, one can obtain a non-trivial upper bound for the Erd{\H o}s unit distance problem by taking $\mathcal{F}\subset \mathbb{R}^2$ a set of points, and $\mathcal{G}$ the set of unit circles centred at these points. Then $G(\mathcal F,\mathcal G)$ contains no $K_{2,3}$, and the number of edges in $G(\mathcal F,\mathcal G)$ can be upper-bounded with the K\H ov\'ari--S\'os--Tur\'an theorem (see \cite[Chapter~4.5]{M02}).

To study probabilities of the form $\Pr(X\in S)$ using incidence graphs, we consider a partition $\{1,\dots,n\}=I_1\sqcup I_2$, and defining the random variables $X_i=\sum_{j\in I_i}a_j\xi_j$ (where $\xi_1,\dots,\xi_n$ are i.i.d.\ Rademacher random variables), we let $V_i$ be the support of $X_i$. Then we consider the bipartite graph $G$ with parts $V_1,V_2$
and an edge between $x\in V_1$ and $y\in V_2$ whenever $x+y\in S$. Note that $G$ can be interpreted as the bipartite incidence graph $G(\mathcal F,\mathcal G)$ for $\mathcal{F}=V_1$ and $\mathcal{G}=\{S-x:x\in V_2\}$. By assigning each vertex $x\in V_i$ the weight $w(x):=\Pr(X_i=x)$, we have
$$\Pr(X\in S)=\Pr(X_1+X_2\in S)=\sum_{xy\in E(G)}w(x)w(y).$$
Now, under certain geometric conditions on $S$, we can show that $G$ is $K_{a,b}$-free for some $a,b$. We can then bound $\Pr(X\in S)$ by proving and applying a weighted version of the K\H ov\'ari--S\'os--Tur\'an theorem (\cref{lem:weighted-KST}) to $G$. Specifically, we note that $G$ is $K_{a,b}$-free if either of the following two equivalent conditions hold.
\begin{enumerate}
    \item (Additive combinatorics interpretation). For any $A,B\subset \mathbb{R}^d$ with $|A|=a$, $|B|=b$, we have $A+B\not \subseteq S$.
    \item (Incidence geometry interpretation). For any $A\subset \mathbb{R}^d$ with $|A|=a$, we have $\left|\bigcap_{a\in A}(S-a)\right|<b$.
\end{enumerate}
As an application of these ideas, suppose that $S\subset \mathbb{R}^d$ has dimension $\dim(S)\le d-1$. Then, heuristically, we would expect that $d$ generic translates $S-x_1,S-x_2,\dots,S-x_d$ intersect in a $0$-dimensional set, i.e., finitely many points. If there is a constant $D$ such that for \emph{every} choice of translates, we have $|(S-x_1)\cap\dots\cap(S-x_d)|\le D$, then $G$ is $K_{d,D+1}$-free. To prove an optimal anti-concentration bound for $\Pr(X\in S)$, we will need a version of this condition which is inherited by intersections with affine subspaces.

\begin{defn}\label{def:generic}%
Say that $S\subseteq \mathbb{R}^d$ satisfies the \emph{generic intersection property} if there is $D\in \NN$ such that the following holds. For any affine subspace $\Lambda\subseteq \mathbb{R}^d$ (including $\Lambda=\mathbb{R}^d$), we have
$$\left|\bigcap_{i=1}^{\dim \Lambda}\left((S\cap \Lambda)-x_i\right)\right|\le D$$
for any distinct $x_1,\dots,x_{\dim \Lambda}\in \mathbb{R}^d$.
\end{defn}

Using the strategy described above together with a strong Littlewood--Offord-type theorem due to H\'alasz~\cite{Hal77} (see \cref{thm:Halasz}), we are able to prove \cref{conj:forwards} for sets $S$ satisfying the generic intersection property.
\begin{thm}
\label{thm:variousS}
Suppose that $S\subset \mathbb{R}^d$ has the generic intersection property. Consider nonzero $d$-dimensional vectors $a_{1},\dots,a_{n}\in\RR^{d}$, and write
$X=a_{1}\xi_{1}+\dots+a_{n}\xi_{n}$, where $\xi_{1},\dots,\xi_{n}$
are i.i.d.\ Rademacher random variables. Then there is $C_S$ depending only on $S$ such that:
\begin{compactenum}
    \item $\Pr(X\in S)\le C_Sn^{-1/2}$, and
    \item $\Pr(X\in S)\le C_S\rho^{1/(d+1)}$, where $\rho=\max_{x\in\mathbb{R}^d}\Pr(X=x)$.
\end{compactenum}
\end{thm}
The first part of \cref{thm:variousS} is best possible, but the second part may not be. Some interesting examples of sets $S$ satisfying the generic intersection property are strictly convex plane curves, irreducible complex algebraic curves other than lines, and 2-dimensional spheres. Also, note that if $S$ can be represented as a finite union of sets satisfying \cref{conj:forwards}, then $S$ satisfies \cref{conj:forwards} itself. So, \cref{thm:variousS} actually implies that \cref{conj:forwards} holds for plane curves with any finite number of inflection points and vertical tangent lines. 

As a final remark, we note that there are very close connections between our weighted K\H ov\'ari--S\'os--Tur\'an theorem (\cref{lem:weighted-KST}) and some other inequalities previously used in connection with the Littlewood--Offord problem. For example, in \cite{Cos13}, a weighted version of the closely-related Szemer\'edi--Trotter theorem~\cite{IKRT06} was applied to study the quadratic Littlewood--Offord problem. There is also a close connection to \emph{decoupling inequalities} (see for example \cite{Cos13,CTV06,RV13}) commonly applied to the polynomial Littlewood--Offord problem, which are themselves closely related to \emph{Sidorenko's conjecture} (see for example \cite{Sid93}) in extremal graph theory. See \cref{rem:decoupling} for further discussion.

\subsection{Sum hypergraphs, and an application to points in convex position}\label{subsec:hypergraphs}
In \cref{subsec:graphs}, we explained how to prove bounds on probabilities of the form $\Pr(X\in S)$ under the assumption that $S$ does not include a sumset $A+B$, for small sets $A,B$. In this subsection we explain how to generalise this to prove \cref{thm:relative}, giving bounds on probabilities of the form $\Pr(X\in S)$ under the assumption that $S$ does not include an iterated sumset $A_1+\dots+A_k$. The ideas in this subsection will be extended further in the next subsection, when we begin to discuss the geometry of sets definable with respect to an o-minimal structure.

We remark that some other natural questions on the additive combinatorics of finite subsets of $S\subset \mathbb{R}^d$ can be phrased in terms of the intersection of $S$ with sumsets $A_1+\dots+A_k$. For example, if $k=d$ and $A_i$ consists of multiples of the $i$th standard basis vector, then $A_1+\dots+A_k$ is a \emph{combinatorial box} (which may also be expressed in the form $X_1\times \dots \times X_d\subseteq \mathbb{R}^d$). The \emph{Elekes--R\'onyai--Szab\'o problem}~\cite{ER00,ES12} asks what structure is imposed on $S$ if $S\subseteq \mathbb{R}^d$ is the zero locus of a single polynomial and has large intersection with a combinatorial box $X_1\times \dots \times X_d\subset \mathbb{R}^d$. See \cite{RSZ16,RSZ18,RSS16,RS20} for some recent work on this problem.
If we specialise further, requiring $X_i=[a_i,b_i]\cap \frac{1}{m}\mathbb{Z}$ for some fixed $m$, then the size of the intersection with $S$ determines the number of rational points of $S$ with denominator $m$ in a box $\prod_i[a_i,b_i]$. This direction of study has been closely connected to recent developments in \emph{o-minimal geometry} (see \cref{subsec:o-minimal}).

As an example of a natural geometric condition which precludes the inclusion of an iterated sumset, it is known that if a set of points in $\RR^d$ lies in convex position (meaning no point is a convex combination of the others), then it cannot contain a sumset of the form $A_1+\dots+A_{d+1}$, where each $|A_i|\ge 2$ (see \cref{lem:Zonotope-bound-weak}).

\begin{defn}\label{def:incidence-hypergraph}
For $S\subset \mathbb{R}^d$ and subsets $V_1,\dots,V_k\subset \mathbb{R}^d$, we define the $k$-uniform $k$-partite \emph{sum hypergraph} $H_S(V_1,\dots,V_k)$ to have vertex sets $V_1,\dots,V_k$, and an edge $v_1\dots v_k$ whenever $v_1+\dots+v_k\in S$.

For $a_1,\dots,a_n\in \RR^d$, and a partition $\{1,\dots,n\}=I_1\sqcup \dots \sqcup I_k$, we define the random variables $X_i=\sum_{j\in I_i}a_j\xi_j$, where $\xi_1,\dots,\xi_n$ are i.i.d.\ Rademacher random variables, and let $V_i$ be the support of $X_i$. Then we define $H_{X,S}(I_1,\dots,I_k)$ to be the hypergraph $H_S(V_1,\dots,V_k)$, with associated vertex-weight function $w$ given by $w(x):=\Pr\left(X_i=x\right)$ for $x\in V_i$.
\end{defn}

Note that with this definition, we have
$$\Pr(X\in S)=\Pr(X_1+\dots+X_k\in S)=\sum_{x_1\dots x_k\in E(H_{X,S}(I_1,\dots,I_k))}w(x_1)\dots w(x_k).$$
Erd{\H o}s~\cite{Erd64} proved a hypergraph extension of the K\H ov\'ari--S\'os--Tur\'an theorem, which in particular gives an upper bound on the number of edges in a $k$-uniform hypergraph which is $K^{(k)}_{2,\dots,2}$-free. We prove a weighted version of this (actually a slight generalisation; see \cref{cor:hypercomprho}) and we use this to prove \Cref{thm:relative}. As an application we prove the following theorem on sets of points in convex position. 

\begin{thm}
\label{thm:convex}
Suppose that $S\subset \mathbb{R}^d$ is a set of points in convex position. Consider nonzero $d$-dimensional vectors $a_{1},\dots,a_{n}\in\RR^{d}$, and write
$X=a_{1}\xi_{1}+\dots+a_{n}\xi_{n}$, where $\xi_{1},\dots,\xi_{n}$
are i.i.d.\ Rademacher random variables. Then there is $C_d$ depending only on $d$ such that:
\begin{compactenum}
    \item $\Pr(X\in S)\le C_dn^{-d/2^d}$, and
    \item $\Pr(X\in S)\le C_d\rho^{1/(d2^{d-1})}$, where $\rho=\max_{x\in\mathbb{R}^d}\Pr(X=x)$.
\end{compactenum}
\end{thm}

If $S$ is any affine variety not containing a line segment, then one can show using B\'ezout's theorem that for sufficiently large $k$ there is no sumset of the form $A_1+\dots+A_k$ included in $S$, where each $|A_i|\ge 2$.
However, we instead consider a more sophisticated variation of this fact, in the context of sets definable with respect to o-minimal structures, in the following subsection.

\subsection{o-minimal geometry and irreducible components}\label{subsec:o-minimal}

Recall that affine varieties in $\RR^d$ have notions of ``irreducible component'' and ``degree'', which control the complexity of intersections. More precisely, the following properties are satisfied.
\begin{enumerate}
    \item Every $k$-dimensional affine variety $S\subset \mathbb{R}^d$ with degree $r$ is a union of at most $r$ irreducible varieties with dimension at most $k$.
    \item (B\'ezout's Theorem) Every pair of distinct irreducible $k$-dimensional varieties $S_1,S_2\subset \mathbb{R}^d$ of degree at most $r$ intersect in an affine variety with dimension at most $k-1$ and degree at most $r^2$.
\end{enumerate}
Recalling the $k$-uniform sum hypergraph $H_{S}(V_1,\dots,V_k)$ from \cref{def:incidence-hypergraph}, the combinatorial manifestation of the above facts is that if $S$ is an affine variety not containing a line segment then this hypergraph satisfies a certain recursive property (we note that the hypothesis of not containing a line segment ensures that no irreducible variety is equal to a nontrivial translate of itself). In a hypergraph $H$ with edge set $E$, recall that the \emph{link} of a vertex $v$ is the $(k-1)$-uniform hypergraph with edge set $\{e:e\cup \{v\}\in E\}$. The \emph{common link} of a pair of vertices is the intersection of the links of those vertices.
\begin{defn}
\label{def:hypergraph-complexity}
Say that a 1-uniform hypergraph $H$ has \emph{complexity} $M$ if it has at most $M$ edges. For $k>1$, say that a $k$-uniform $k$-partite hypergraph $H$ with vertex sets $V_1,\dots,V_k$ has complexity $M$ if we can write $H=H_1\cup \dots H_M$ such that for $1\le i \le M$, the common link in $H_i$ of any two elements $x,y\in V_1$ is a $(k-1)$-uniform $(k-1)$-partite hypergraph on the vertex sets $V_2, \dots, V_k$ with complexity $M$.
\end{defn}

Note that if a $k$-uniform $k$-partite hypergraph $H$ is $K^{(k)}_{2,\dots,2}$-free, then it has complexity $1$. As it happens, a bounded-complexity assumption is enough to prove a (weighted) K\H ov\'ari--S\'os--Tur\'an-type theorem (\cref{cor:hypercomprho}), which implies \cref{thm:baby-forwards,thm:pointvsS} for affine varieties. Actually, one can prove that the sum hypergraph $H$ has bounded complexity under a much weaker condition on $S$, namely the condition that $S$ does not contain a line segment and is \emph{definable with respect to an o-minimal structure}.
We will introduce this notion after discussing some recent developments in the area (see \cite{vdD98} for a historical introduction and \cite{S17} for a more recent survey).

Following earlier results for curves in $\RR^2$ due to Bombieri and Pila~\cite{BP89}, Pila \cite{P95} proved that  any affine variety $S\subset \mathbb{R}^d$ of degree $r$ and dimension $k$ satisfies
$$\left|S\cap \frac{1}{n}\mathbb{Z}^d\cap \prod_{i=1}^d [a_i,b_i]\right|=n^{k-1+1/r+o(1)}$$
for any constants $a_1<b_1,\dots,a_d<b_d$. Extending beyond the algebraic setting, Pila and Wilkie \cite{PW06} considered the setting where $S\subset \mathbb{R}^d$ is definable with respect to an o-minimal structure, in which case one can obtain similar results by decomposing $S$ into an ``algebraic part''
and a complementary ``transcendental part''. 
Exploiting these rational point counting theorems, powerful results in arithmetic dynamics such as the Manin--Mumford conjecture and parts of the Andr\'e--Oort conjecture have recently been attacked with great success (see for example \cite{P11,PU08}). 

On the more combinatorial side, a number of results bounding the number of edges in incidence graphs and hypergraphs (such as the work of Fox, Pach, Sheffer, Suk and Zahl~\cite{FPSSZ17}, and Do~\cite{Do18} on the semi-algebraic Zarankiewicz problem) have been generalised in the o-minimal direction. In particular, Basu and Raz \cite{BR18} have generalised the  Szemer\'edi--Trotter theorem to certain o-minimal incidence graphs and Chernikov, Galvin, and Starchenko \cite{CGS20} have generalised this to more general incidence hypergraphs.

To proceed further, we finally define the notion of an o-minimal structure. Recall that a first-order formula over $\mathbb{R}$ is a well-formed sentence involving parentheses $(,)$, quantifiers $\exists,\forall$, predicates (such as $>,<,=$), connectives (such as $\vee,\wedge,\rightarrow,\lnot,\top,\bot$), variables (which we give single-letter names such as $x,y,z$), operations (such as $+$, $\times$, $\exp$), and  constants (one for every element of $\mathbb{R}$), such that we can specialise the unbound variables to elements of $\mathbb{R}$ and interpret the sentence as either true $\top$ or false $\bot$. For example, if $\phi(x)$ is the formula $(\exists y)(y^2=x)\wedge (\exists z)(z^2=1-x)$, then $\phi(x)$ is true exactly when the unbound variable $x$ has $x\in[0,1]$.

\begin{defn}
A collection $\mathcal{F}$ of functions $f:\mathbb{R}^{\ell_f}\to \mathbb{R}$ of various arities $\ell_f$ generates an \emph{o-minimal structure $\mathbb{R}_\mathcal{F}$} (over the real field) if any first-order formula $\phi$ involving one free parameter $x$, operations in $\mathcal{F}\cup \{+,\times\}$ and predicates $>,<,=$, has the property that the set
$$S_{\phi}:=\{x:\phi(x)\text{ is true}\}\subseteq \mathbb{R}$$
is a finite union of points and (possibly unbounded) intervals.

We say a subset $S\subseteq \mathbb{R}^d$ is \emph{definable} (with respect to $\mathcal{F}$) if there is a first-order formula $\phi$ with $d$ free parameters $x_1,\dots,x_d$, operations in $\mathcal{F}\cup \{+,\times\}$, and predicates $>,<,=$, such that $$S=S_{\phi}:=\{(x_1,\dots,x_d): \phi(x_1,\dots,x_d)\text{ is true}\}\subseteq \mathbb{R}^d.$$
\end{defn}
\textbf{Example 0:} $\mathcal{F}=\{f\}$ with $f(x,y)=\sin(x+y)$ does not generate an o-minimal structure. Indeed, the formula $\phi(x)$ given by $f(x,0)=0$ has
$S_{\phi}=2\pi \mathbb{Z}$,
which is not a finite union of points and intervals.

\textbf{Example 1:} Take $\mathcal{F}=\emptyset$. Obviously, semi-algebraic sets (Boolean combinations of subsets of $\mathbb{R}^d$ given by polynomial equalities and inequalities) are definable. \textit{A priori}, it may seem that other more exotic sets are definable as well (for example, images of semi-algebraic sets under polynomial maps are definable, and one can define sets using functions that are themselves only defined implicitly in terms of polynomials). However, the Tarski--Seidenberg theorem \cite{T51} states that in this setting actually the only definable sets are semi-algebraic subsets of $\mathbb{R}^d$ (so, the ``exotic'' sets described above are actually semi-algebraic, though this isn't obvious), and taking $d=1$ shows that $\mathcal{F}$ generates an o-minimal structure (commonly denoted $\RR_{\mathrm{alg}}$).

\textbf{Example 2:} Wilkie~\cite{W96} showed that $\mathcal{F}=\{\exp\}$ generates an o-minimal structure. In this setting, definable sets are rather more complicated to describe, 
but in \cite{W96} it is shown that they are all of the form $\pi(\{(x_1,\dots,x_n):f(x_1,\dots,x_n,e^{x_1},\dots,e^{x_n})=0\})$, where $f$ is a polynomial in $2n$ variables and $\pi$ is a linear projection $\mathbb{R}^n\to \mathbb{R}^m$ for some $m\in \NN$. 

\textbf{Example 3:}
van den Dries and Miller \cite{DM94} showed that $\exp$, together with all ``restricted analytic functions'' ($f|_{[0,1]^d}$ for $f$ an analytic function on an open neighborhood of $[0,1]^d$) generate an o-minimal structure.

In the setting of o-minimal geometry, the following notion plays a similar role to the notion of degree for affine varieties, and description complexity for semi-algebraic sets (recall \cref{def:description-complexity}).

\begin{defn}\label{def:o-complexity}
We say an o-minimal structure $\mathbb{R}_{\mathcal{F}}$ is said to be \emph{finitely generated} if $\mathcal{F}$ is finite.

For a finitely generated o-minimal structure $\mathbb{R}_{\mathcal{F}}$, the \emph{complexity} of a formula $\phi$ is the number of symbols used in its description (every constant from $\mathbb{R}$ counts as one symbol). The complexity of a definable set $S$ is then defined to be the minimum complexity of a formula defining $S$.
\end{defn}
Examples 1 and 2 above describe finitely generated o-minimal structures, but Example 3 does not. 
For an arbitrary o-minimal structure $\mathbb{R}_{\mathcal{F}}$, note that every definable set only uses a finite number of the functions in $\mathcal F$ in its description. Hence restricting our attention to finitely generated o-minimal structures is not actually a restriction at all, and in fact our notion of complexity is equivalent to the notion of ``definable families'' more commonly used in the o-minimality literature (see \Cref{rem:complexity}).

By the Tarski--Seidenberg theorem (see \cite{T51}), the degree of an affine variety and the description complexity of a semi-algebraic set are bounded from above and below by a function of its complexity with respect to the o-minimal structure $\RR_{\emptyset}=\RR_{\mathrm{alg}}$.

We show that the above notion of complexity controls the hypergraph notion of complexity in \cref{def:hypergraph-complexity}. This is done by constructing an analogue of irreducible components we call \emph{self-irreducible components}, for definable sets containing no line segment, which satisfies a restricted analogue of B\'ezout's theorem in which we only consider intersections $S_1\cap S_2$ for which $S_1,S_2$ are translates of each other.

\begin{prop}\label{prop:o-minimal-complexity}
For a finitely generated o-minimal structure $\RR_{\mathcal F}$ and any $r,d\in \NN$, there is $C_{\mathcal F,d,r}$ such that for any $k$-dimensional definable subset $S\subset \mathbb{R}^d$ with complexity $r$ not containing a line segment, any $(k+1)$-uniform $(k+1)$-partite intersection hypergraph $H_{S}(V_1,\dots,V_{k+1})$ has complexity at most $C_{\mathcal F,d,r}$. 
\end{prop}
Using \cref{prop:o-minimal-complexity} and a K\H ov\'ari--S\'os--Tur\'an-type theorem, we deduce the following strong bound comparing $\Pr(X\in S)$ with the maximum point concentration probability.

\begin{thm}
\label{thm:pointvsS}
Let $S\subseteq\RR^d$ be a set which is definable with respect to an o-minimal structure and does not contain any line segment. Consider nonzero $d$-dimensional vectors $a_{1},\dots,a_{n}\in\RR^{d}$, and write
$X=a_{1}\xi_{1}+\dots+a_{n}\xi_{n}$, where $\xi_{1},\dots,\xi_{n}$
are i.i.d.\ Rademacher random variables. Then there is a constant $C_S$ such that for $k=\dim(S)$ and $\rho=\max_{x\in \RR^d}\Pr(X=x)$, we have
$$\Pr(X\in S)\le C_{S}\rho^{1/((k+1)2^k)}.$$
\end{thm}
Here the constant $C_S$ depends only on the dimension $d$ and the complexity of $S$.

The proof of \cref{thm:baby-forwards} is a little more involved, and combines most of the ideas we have discussed so far. First, using the Tao--Vu inverse theorem we observe that the desired result follows from \Cref{thm:pointvsS} unless most of the coefficients $a_1,\dots,a_n$ lie in a small generalised arithmetic progression with low rank. After conditioning on a small number of $\xi_i$, the random variable $X$ is conditionally supported in a low-rank generalised arithmetic progression $Q$. We then adapt parts of a general theory due to Pila~\cite{P09} to cover $S\cap Q$ with a small number of projections of open subsets of semi-algebraic sets. Although the resulting covering sets are not themselves semi-algebraic, we can use similar ideas as in the proof of \Cref{thm:Gotsman} to deduce \cref{thm:baby-forwards}.

We note that in the works mentioned earlier in this subsection, counting the number of rational points \cite{PW06} and algebraic points \cite{P09} of bounded height on definable sets $S$, one has very poor control on the number of such points lying on the ``algebraic'' part of $S$. One therefore obtains the strongest bounds by excising this part from $S$. It is worth highlighting that our situation is quite different: we have bounds of the form $\Pr(X\in S_\mathrm{alg})\le n^{-1/2+o(1)}$ for the algebraic part $S_\mathrm{alg}$, nearly matching the point concentration bound $\Pr(X=x)=O(n^{-1/2})$ given by the Erd\H os--Littlewood--Offord theorem.

\section{Preliminaries on point concentration}\label{sec:point}

In this section we collect a few results about point concentration, which will be useful throughout the paper. For $a_1,\dots,a_n\in \RR^d$, let $\rho(a_1,\dots,a_n)=\max_{x\in \RR^d}\Pr(a_1\xi_1+\dots+a_n\xi_n=x)$, where $\xi_1,\dots,\xi_n$ are i.i.d. Rademacher random variables. We start with the following simple facts.

\begin{fact}\label{fact:concentration-monotone}
For any subset $I\subseteq \{1,\dots,n\}$ we have $\rho(a_1,\dots,a_n)\le \rho((a_i)_{i\in I})$.
\end{fact}


\begin{fact}\label{fact:concentration-product}
For any subset $I\subseteq \{1,\dots,n\}$ we have $\rho(a_1,\dots,a_n)\ge \rho((a_i)_{i\in I})\rho((a_i)_{i\notin I})$.
\end{fact}


\begin{fact}\label{fact:concentration-lipschitz}
We have $\rho(a_1,\dots,a_n)\ge \rho(a_1,\dots,a_{n-1})/2$.
\end{fact}


We also need a ``partite'' notion of point concentration.
\begin{defn}\label{def:partite-concentration}
Say that a sequence $(a_1,\dots,a_n)\in (\RR^d)^n$ is \emph{$(\lambda_1,\dots,\lambda_m)$-anti-concentrated} if there is a partition $I_1\sqcup \dots\sqcup I_m=\{1,\dots,n\}$ such that $\max_{x\in \RR^d}\Pr(\sum_{i\in I_j}a_i\xi_i=x)\le \lambda_j$ for all $j$. If $\lambda_1=\dots\lambda_m=\lambda$, we abbreviate this by saying $(a_1,\dots,a_n)$ is \emph{$m$-part $\lambda$-anti-concentrated}.
\end{defn}
For example, note that if each $a_i\ne 0$ then $(a_1,\dots,a_n)$ is $m$-part $O(m^{1/2}n^{-1/2})$-anti-concentrated. Indeed, this follows from the Erd\H os--Littlewood--Offord theorem, and any equipartition $I_1\sqcup \dots\sqcup I_m$ will do.

We next show that upper bounds on point concentration imply similar upper bounds on partite point concentration.
\begin{lem}\label{lem:part-concentration}
For $0\le \lambda_1,\dots,\lambda_m< 1/2$, if $\rho(a_1,\dots,a_n)\le \prod_{i=1}^m\lambda_i$,  then $(a_1,\dots,a_n)$ is  $(2\lambda_1,\dots,2\lambda_m)$-anti-concentrated.
\end{lem}
\begin{proof}
We prove this by induction on $m$, noting that the case $m=1$ is trivial. Let $p_i= \rho({a_1,...,a_i})$ and let $q_i = \rho({a_{i+1},...,a_n})$. Note that $p_0 = 1$ and $p_n=\rho(a_1,\dots,a_n)$.

Now, for $i = 0,1,\dots,n-1$, we have $p_i \ge p_{i+1} \ge p_i /2$ by \cref{fact:concentration-monotone} and \cref{fact:concentration-lipschitz}. There exists by discrete continuity an $i$ for which $\lambda_1 \le p_i \le 2\lambda_1$. Note that $p_iq_i\le \rho(a_1,\dots,a_n)\le \prod_{i=1}^m\lambda_i$ by \cref{fact:concentration-product}, so $q_i\le \prod_{i=2}^m\lambda_i$. Therefore, by the induction hypothesis, $(a_{i+1},\dots,a_n)$ is $(2\lambda_2,\dots,2\lambda_m)$-anti-concentrated (with partition $I_2\sqcup\dots\sqcup I_m=\{i+1,\dots,n\}$, say). Taking $I_1=\{1,\dots,i\}$, the partition $I_1\sqcup \dots\sqcup I_m=\{1,\dots,n\}$ shows that $(a_1,\dots,a_n)$ is $(2\lambda_1,\dots,2\lambda_m)$-anti-concentrated, as desired.
\end{proof}

Finally, we will need the following theorem of Hal\'asz, which gives strong bounds for the Littlewood--Offord problem in the case where the coefficients ``robustly'' span the entire space $\RR^d$.

\begin{thm}[{\cite[Theorem~1]{Hal77}}]\label{thm:Halasz}
Let $a_1,\dots,a_n\in \RR^d$ be vectors such that no proper linear subspace of $\RR^d$ contains half of the vectors $a_i$. Let $\xi_1,\dots,\xi_n$ be i.i.d.\ Rademacher random variables, and let $X=a_1\xi_1+\dots+a_n\xi_n$. Then $\max_x\Pr(X=x)\le C_d n^{-d/2}$, where $C_d$ depends only on $d$.
\end{thm}

\section{A weighted version of the K\H ov\'ari--S\'os--Tur\'an theorem}
In this section we prove a weighted version of the K\H ov\'ari--S\'os--Tur\'an theorem, which we will use frequently to bound the total weight of various graphs and hypergraphs.
\begin{thm}\label{lem:weighted-KST}Consider a bipartite graph $G$ with bipartition $A\cup B$, where the vertices
in $A,B$ are weighted (with total weights $w(A)=w(B)=1$). Let $w(G)=\sum_{a\sim b}w(a)w(b).$
Assume every set of $t$ vertices in $A$ has common neighbourhood
with total weight at most $q$, and assume every vertex in $A$ has weight at most $\rho$. Then
\[
w(G)\le q^{1/t}+t\rho.
\]
\end{thm}

We remark that the unweighted K\H ov\'ari--S\'os--Tur\'an theorem  follows from \cref{lem:weighted-KST} if we set $w(a)=1/|A|$ for each $a\in A$ and $w(b)=1/|B|$ for each $b\in B$.

\begin{rem}\label{rem:decoupling}
A weighted bipartite graph $G$ as in \cref{lem:weighted-KST} can be viewed as encoding an event $\mathcal E(X_1,X_2)$ depending on two random variables $X_1,X_2$. Given independent copies $X_1^{(1)},\ldots,X_1^{(t)}$ of $X_1$, it follows from H\"older's inequality (or alternatively Jensen's inequality) that
$$\Pr(\mathcal E(X_1,X_2))\le \Pr(\mathcal E(X_1^{(1)},X_2)\cap \ldots \cap \mathcal E(X_1^{(t)},X_2))^{1/t}.$$
This can be interpreted as a \emph{decoupling inequality} (see for example \cite[Lemma~4.7]{CTV06} for a related inequality). Let $\mathcal F$ be the event that some $X_1^{(i)}=X_1^{(j)}$, so in the setting of \cref{lem:weighted-KST},
\begin{align*}
    \Pr(\mathcal E(X_1^{(1)},X_2)\cap \ldots \cap \mathcal E(X_1^{(t)},X_2))\le \Pr(\mathcal E(X_1^{(1)},X_2)\cap \ldots \cap \mathcal E(X_1^{(t)},X_2)\,|\,\overline{\mathcal F})+\Pr(\mathcal F)\le q+\binom t 2\rho,
\end{align*}
and therefore $w(G)\le \left(q+\binom t 2\rho\right)^{1/t}$. This is comparable to the bound $q^{1/t}+t\rho$ in \cref{lem:weighted-KST}, unless $\rho$ is substantially larger than $q$. We note that this weaker inequality is sufficient for all applications in this paper, except the second part of \cref{thm:variousS}.
\end{rem}

To prove \cref{lem:weighted-KST}, first we prove a version where $A$ is not weighted.

\begin{lem}\label{lem:weighted-KST-weak}Consider a bipartite graph $G=(A,B)$ where the vertices
in $B$ are weighted (with total weight 1). Let $w(G)=\sum_{a\sim b}w(b).$
Assume every set of $t$ vertices in $A$ has common neighbourhood
with total weight at most $q$. Then
\[
w(G)\le q^{1/t}n+t-1.
\]
\end{lem}

\begin{proof}We assume that $w(G)\ge t-1$, or else we are immediately done. For a vertex $b\in B$, let $d_{b}$ be the degree of $b$. We
have $\sum_{b\in B}w(b)\binom{d_{b}}{t}=\sum_{a_{1},\dots,a_{t}}w(N(a_{1},\dots,a_{t}))\le q\binom{n}{t}$,
where the latter sum is over all $t$-sets of distinct vertices $a_{1},\dots,a_{t}\in A$. Now, the function $x\mapsto \binom x t$ is convex for $x\ge t-1$ so by Jensen's inequality we have \[
\binom{w(G)}{t}=\binom{\sum_{b\in B}w(b)d_b}{t}\le q\binom{n}{t}.
\]
Observing that $\binom{w(G)}t>(w(G)-t+1)^t/t!$ and $\binom n t<n^t/t!$, the desired result follows.\end{proof}

Now we deduce \cref{lem:weighted-KST}.

\begin{proof}[Proof of \cref{lem:weighted-KST}]
Note that $w\left(G\right)$ can be viewed as a linear function of
$w\in\left[0,\rho\right]^{A}$. So, the maximum value of $w\left(G\right)$,
under the constraint $w(A)=1$ and holding $w|_{B}$ constant, is attained when
$w\left(u\right)\in\left\{ 0,\rho\right\} $ for all but one $u\in A$.

We can therefore assume that after deleting all the weight-zero vertices in $A$, and possibly one additional vertex, and multiplying all weights in $A$ by $1/\rho$, we arrive at a graph $G'$ with parts $A'$ and $B$, where each of the $|A'|=\floor{1/\rho}$ vertices in $A'$ are unweighted, and the vertices in $B$ have the same weights as in $G$.

By \cref{lem:weighted-KST-weak}, we have $w(G')\le q^{1/t}(1/\rho)+t-1$. Multiplying by $\rho$ and observing that the deleted
vertices can contribute a weight of at most $\rho$,
the desired result follows.
\end{proof}

\section{Sets with the generic intersection property}
In this section we prove \cref{thm:variousS}.
 The following proposition will actually imply both parts of \cref{thm:variousS}.

\begin{prop}\label{prop:rel-GIP}
Let $S\subseteq \RR^d$ have the property that every $t$ translates of $S$ intersect in at most $M$ points. Suppose $(a_1,\dots,a_n)$ is $(\lambda_1,\lambda_2)$-anti-concentrated. Then for $X=a_1\xi_1+\dots+a_n\xi_n$, where $\xi_1,\dots,\xi_n$ are independent Rademacher random variables, we have $\Pr(X\in S)\le M^{1/t}\lambda_2^{1/t}+t\lambda_1$.
\end{prop}

\begin{proof}
Let $\{1,\dots,n\}=I_1\sqcup I_2$ be a partition witnessing the fact that $(a_1,\dots,a_n)$ is  $(\lambda_1,\lambda_2)$-anti-concentrated, and let $H=H_{X,S}(I_1,I_2)$ be the bipartite sum graph from \cref{def:incidence-hypergraph}.

Now, the common neighbourhood of a $t$-tuple of vertices $x_1,\dots,x_t\in A$ is the set of all $y\in B$ lying in the intersection $(S-x_1)\cap \dots\cap (S-x_t)$, so by assumption this common neighbourhood has at most $M$ vertices, and therefore has total weight at most $M\lambda_2$. The desired result then follows from  \cref{lem:weighted-KST}.
\end{proof}

We now prove \cref{thm:variousS}.

\begin{proof}[Proof of \cref{thm:variousS}]
Recall that $S$ having the generic intersection property means that there is a constant $D_S$ such that for any affine subspace $\Lambda\subset \mathbb{R}^d$ (including $\Lambda=\mathbb{R}^d$) and distinct $x_1,\dots,x_{\dim \Lambda}\in \mathbb{R}^d$ we have
$$\left|\bigcap_{i=1}^{\dim \Lambda}\left((S\cap \Lambda)-x_i\right)\right|\le D_S.$$

For the first part of the theorem, we show that $\Pr(X\in S)=O(n^{-1/2})$, where the implicit constant only depends on $d,D_S$. We induct on $d$. The result for $d=1$ follows from the Erd{\H o}s--Littlewood--Offord theorem, so assume $d\ge 2$ and that the result is true for all dimensions less than $d$.

Let $C_d$ be as in \Cref{thm:Halasz}. First, if $(a_1,\dots,a_n)$ is $2$-part $\lambda$-anti-concentrated, for $\lambda=C_d/\floor{n/2}^{-d/2}=O(n^{-d/2})$, then we can apply \cref{prop:rel-GIP} to obtain $\Pr(X\in S)=O((n^{-d/2})^{1/d})=O(n^{-1/2})$ as desired.

Otherwise, if $\rho(a_1,\dots,a_{\floor{n/2}})> C_d/\floor{n/2}^{-d/2}$ or $\rho(a_{\floor{n/2}+1},\dots,a_{n})> C_d/\floor{n/2}^{-d/2}$ then by \Cref{thm:Halasz} at least $\floor{n/2}/2$ of the coefficients $a_i$ lie in a linear subspace $W\subset \RR^d$ of dimension $d-1$ (without loss of generality, suppose this is the case for the first  $\floor{n/2}/2$ coefficients).
Let $X_1=\sum_{i\le \lfloor n/2 \rfloor/2} a_i\xi_i$ and $X_2=\sum_{i>\lfloor n/2 \rfloor/2} a_i\xi_i$. Then
$$\Pr(X\in S)=\Pr(X_1\in S-X_2)=\Pr(X_1\in (S-X_2)\cap W).$$ But, if we condition on any outcome of $X_2$, the set $(S-X_2)\cap W \subseteq W$ itself satisfies the generic intersection property (inside $W\cong \RR^{d-1}$), so by the inductive hypothesis we know that
$$\Pr(X_1\in (S-X_2)\cap W)=O((\lfloor n/2 \rfloor/2)^{-1/2})=O(n^{-1/2}),$$
implying the desired result.

For the second part of the theorem, we have to show that $\Pr(X\in S)=O(\rho^{1/(d+1)})$, where $\rho=\max_{x\in\mathbb{R}^d}\Pr(X=x)$. To do this we observe that $(a_1,\dots,a_n)$ is $(2\rho^{1/(d+1)},2\rho^{d/(d+1)})$-anti-concentrated, by \Cref{lem:part-concentration}, and then we apply 
\cref{prop:rel-GIP}. (We remark that here we actually only used the generic intersection property for $\Lambda=\mathbb{R}^d$).
\end{proof}

\section{An extremal theorem for bounded-complexity hypergraphs}
\label{sec:comphypergraph}
Recall the definition of hypergraph complexity from \cref{def:hypergraph-complexity}. We prove the following K\H ov\'ari--S\'os--Tur\'an-type theorem for hypergraphs of bounded complexity, and deduce \cref{thm:relative} from it.
\begin{thm}
\label{thm:hypercomprho}
If $H$ is an $m$-partite $m$-uniform hypergraph with complexity $M$ on the vertex sets $V_1,\dots,V_m$, and there is a vertex weighting function $w$ such that $w(V_j)=1$ and $w(x)\le \lambda$ for all vertices $x$, then
$$w(H)\le C_{M,m} \lambda^{1/2^{m-1}},$$
for some $C_{M,m}$ depending only on $M$ and $m$.
\end{thm}
We remark that there are certain similarities between \cref{thm:hypercomprho} and Costello, Tao and Vu's decoupling lemma~\cite[Lemma~6.3]{CTV06}.
\begin{proof}
We proceed by induction. The result is clearly true for $m=1$, so, we consider some $m\ge 2$ and assume the result is true for uniformities less than $m$.

By the definition of hypergraph complexity, $H$ is a union of at most $C$ hypergraphs $H_i$, in such a way that the common link $N_i(x,y)$ of any $x,y\in V_1$ in $H_i$ is an $(m-1)$-partite $(m-1)$-uniform hypergraph of complexity at most $C$. So, by the inductive hypothesis $w(N_i(x,y))=O(\lambda^{1/2^{m-2}})$. Consider the weighted bipartite graph $G_i$ with vertex sets $V_1$ and $V_2\times \dots \times V_m$, with an edge from $v_1$ to $(v_2,\dots,v_m)$ if $\{v_1,\dots,v_m\}\in E(G_i)$. Define the weighting function $w'$ by $w'(v_1)=w(v_1)$ for all $v_1\in V_1$ and $w'((v_2,\dots,v_m))=w(v_2)\dots w(v_m)$ for all $(v_2,\dots,v_m)\in V_2\times \dots \times V_m$. Then, in $G_i$, the common neighbourhood of $x,y\in V_1$ has weight exactly $w(N_i(x,y))$. Hence we may apply \Cref{lem:weighted-KST} with $\rho=\lambda$, $q=O(\lambda^{1/2^{m-2}})$, and $t=2$, to obtain $$w(H_i)=w'(G_i)=O(\lambda^{1/2^{m-1}}).$$
Summing this over all $i$ gives us the desired bound on $w(H)$.
\end{proof}

We can use the same ideas as in the proof of \cref{thm:variousS} to prove two corollaries on probabilities of the form $\Pr(X\in S)$ (the first of which implies \cref{thm:relative}). For both of these corollaries, we consider nonzero $a_{1},\dots,a_{n}\in\RR^{d}$, and write
$X=a_{1}\xi_{1}+\dots+a_{n}\xi_{n}$, where $\xi_{1},\dots,\xi_{n}$
are i.i.d.\ Rademacher random variables. Recall the definitions of the sum hypergraph $H_{X,S}(I_1,\dots,I_k)$ and $H_S(V_1,\dots,V_m)$ from \cref{def:incidence-hypergraph}.
\begin{cor}
\label{cor:hypercomprho}
Suppose $S\subset \mathbb{R}^d$ is such that any $m$-partite $m$-uniform sum hypergraph of the form $H_{S}(V_1,\dots,V_m)$ has complexity at most $M$. Let $\rho=\max_{x\in\mathbb{R}^d}\Pr(X=x)$. Then we have
$$\Pr(X\in S)\le C_{M,m}\rho^{1/(m2^{m-1})},$$
for some $C_{M,m}$ depending only on $M$ and $m$.
\end{cor}
\begin{proof}
By \Cref{lem:part-concentration}, $(a_1,\dots,a_n)$ is $m$-part $2\rho^{1/m}$-anti-concentrated. Let $I_1\sqcup \dots\sqcup I_m$ be the corresponding partition, and apply \cref{thm:hypercomprho} to $H_{X,S}(I_1,\dots,I_m)$, with $\lambda=2\rho^{1/m}$.
\end{proof}
\begin{proof}[Proof of \Cref{thm:relative}]Recall that we showed assumption (a) implies assumption (b). So, we suppose assumption (b) is satisfied, meaning that $S$ includes no Minkowski sum $A_1+\dots+A_m$ with each $|A_i|=2$. Then, sum hypergraphs of the form $H_S(V_1,\dots,V_m)$ contain no $K^{(m)}_{2,\dots,2}$, meaning that they have complexity 1. Hence, by \cref{cor:hypercomprho} we conclude that $\Pr(X\in S)\le C_{1,m} \rho^{1/(m2^{m-1})}$.
\end{proof}

\begin{cor}
\label{thm:hypercompabs}
Suppose $S$ has the property that for any affine subspace $\Lambda\subseteq \mathbb{R}^d$ with dimension $\ell=\dim(\Lambda)\le d$, any sum hypergraph of the form $H_{X,S\cap \Lambda}(I_1,\dots,I_{\ell})$ has complexity at most $M$. Then
$$\Pr(X\in S)\le C_{M,k,d}n^{-d/2^{d}},$$
for some $C_{M,k,d}$ depending only on $M,d$.
\end{cor}
\begin{proof}
We induct on $d$. In the case $d=1$, the fact that the 1-uniform hypergraph $H_{X,S}(\{1,\dots,n\})$ has complexity at most $M$ implies that $|S\cap X|\le M$, so the desired result follows from the Erd{\H o}s--Littlewood--Offord Theorem. So, consider some $d\ge 2$, and assume the theorem statement holds in dimensions lower than $d$.

Let $C_d$ be as in \Cref{thm:Halasz}, and first suppose $(a_1,\dots,a_n)$ is $d$-part $\lambda$-anti-concentrated, for $\lambda=C_d |I_j|^{-d/2}$. Then, the desired result follows from \Cref{thm:hypercomprho} with $m=d$.

Otherwise, fixing an arbitrary equipartition $\{1,\dots,n\}=I_1\sqcup\dots \sqcup I_d$, there is some $j$ such that $\rho((a_i)_{i\in I_j})\ge C_d |I_j|^{-d/2}$. By \Cref{thm:Halasz}, there is a linear subspace $W\subset \mathbb{R}^d$ of dimension $d-1$ containing half of the $a_i$ in $I_1$ (say $a_1,\dots,a_q$ for some $q\ge \lfloor n/d\rfloor /2$). Let $X_1=a_1\xi_1+\dots+a_q\xi_q$ and $X_2=a_{q+1}\xi_{q+1}+\dots+a_n\xi_n$, so that
$$\Pr(X\in S)=\Pr(X_1\in S-X_2)=\Pr(X_1\in (S-X_2)\cap W).$$
But, if we condition on any outcome of $X_2$, then applying the induction hypothesis to the set $(S-X_2)\cap W \subseteq W$ (lying inside $W\cong \RR^{d-1}$) shows that
$$\Pr(X_1\in (S-X_2)\cap W)=O\left((\lfloor n/d \rfloor/2)^{-(d-1)/2^{d-1}}\right)=O\left(n^{-d/2^{d}}\right),$$
implying the desired result.
\end{proof}
\section{Sets of points in convex position}\label{sec:convex}
In this section we prove \cref{thm:convex}. We start with the following simple fact about sets of points in convex position.
\begin{lem}\label{lem:Zonotope-bound-weak}There is no finite set of points $Z=A_1+\dots+A_{d+1}\subseteq \RR^d$, with $|A_i|=2$, that is in convex position.\end{lem}

\begin{proof}Consider such a set $Z$, and for the purpose of contradiction assume its elements are in convex position. We may assume that the affine span of $Z$ is full-dimensional (otherwise we can induct on the dimension, observing that the $d=0$ case is trivial). First we show that $|Z|=2^{d+1}$. Indeed, if not, let $i$ be the first index such that $|A_1+\dots+A_i|=2^i$ and $|A_1+\dots+A_{i+1}|<2^{i+1}$. Then writing $A_{i+1}=\{x,y\}$, there exist $z,w\in A_1+\dots+A_i$ such that $z+x=y+w$. But then $v_1=w+x$, $v_2=z+x=w+y$ and $v_3=z+y$ are an arithmetic progression (with common difference $z-w=y-x$), and are hence not in convex position. But for any $v\in A_{i+2}+\dots+A_{d+1}$, we have $v_1+v,v_2+v,v_3+v\in Z$, contradicting that the elements of $Z$ are in convex position. 

Now, the convex hull $\conv(Z)$ is a Minkowski sum of intervals $\conv(A_1)+\dots+\conv(A_{d+1})$, also called a \emph{zonotope}. But it is known (see for example \cite{Fuk}) that any zonotope in $d$ dimensions with $m$ generators has at most $2\sum_{i=0}^{d-1}\binom{m-1}{i}$ vertices, and when $m=d+1$ this tells us that $|Z|\le 2^{d+1}-1$, a contradiction.\end{proof}

To obtain the bounds in \Cref{thm:convex} we will actually need the following slight variant of \cref{lem:Zonotope-bound-weak}.

\begin{prop}
\label{prop:convex-zonotope}
If $S\subset \mathbb{R}^d$ is a set of points in convex position, then there exists a cover $S=S_1\cup \dots\cup S_{2d}$ such that for any subsets $A_1,\dots,A_d\subset \mathbb{R}^d$ with $|A_i|=2$, and any $S_j$, we have
$$A_1+\dots+A_d\not\subseteq S_j.$$
\end{prop}

\begin{proof}
Let $S\subseteq \RR^d$ be a set of points in convex position. We need to cover $S$ with $2d$ sets, none of which contains a Minkowski sum $A_1+\dots+A_d$ with each $|A_i|=2$ (call such a Minkowski sum a \emph{bad configuration}).

For each point $x\in S$ let $v_{x}$ be the outward normal vector to a supporting hyperplane $H_x$ for $S$. Then $\langle v_x, y-x\rangle <0$ for every $y\in S\backslash\{ x\} $. Let
$e_{1},\dots,e_{n}$ be the standard basis vectors of $\RR^{d}$,
and let $S_{i}^{+}$ (respectively $S_{i}^{-}$) be the set of $x\in S$
with $\langle v_{x},e_{i}\rangle >0$ (respectively, $\langle v_{x},e_{i}\rangle <0$).
Clearly these sets cover $S$; we will prove that none of them contains a bad configuration.

Without loss of generality, we consider $S_{1}^{+}$. Imagining that
$e_{1}$ points ``downwards'', one should think of $S_{1}^{+}$ as being ``concave up''.
Suppose $Z=\{a_1,b_1\}+\dots+\{a_d,b_d\}$ is a bad configuration. We may assume that
$Z$ does not lie in a $(d-1)$-dimensional affine subspace
(in which case by \cref{lem:Zonotope-bound-weak} it would not be in convex position and
would therefore not be a subset of $S$). It suffices to prove that
there is some point $x\in Z$ which is situated vertically above some
other point $y'=x+\lambda e_1\in\conv(Z)$ (here $\lambda>0$). Indeed, this would imply that $\langle v_{x},y'-x\rangle=\lambda \langle v_x,e_1\rangle > 0$, and hence there is some $y\in Z\subseteq S_1^+$ such that $\langle v_x, y-x\rangle > 0$, contradicting the defining property of $v_x$.

Let $\pi(Z)$ be the projection of $Z$ onto the hyperplane
perpendicular to $e_{1}$, so by \cref{lem:Zonotope-bound-weak}, either some $b_{i}-a_{i}$
is parallel to $e_{1}$ or else there is some $\pi(z)\in\pi(Z)$
which is a convex combination of other points in $\pi(Z)$.
In the former case, the desired fact is obvious, and in the latter
case we may take $x$ to be $z$, if $z$ is above another point of $\conv(Z)$, or the vertex diametrically opposite $z$ in $Z$, otherwise.
\end{proof}
We finally prove \Cref{thm:convex}.
\begin{proof}[Proof of \Cref{thm:convex}]
If $S\subseteq \RR^d$ is a set of points in convex position, then \cref{prop:convex-zonotope} tells us that we can write $S=S_1\cup \dots S_{2d}$ such that any sum hypergraph of the form $H_{S_i}(V_1,\dots,V_d)$ is $K^{(d)}_{2,\dots,2}$-free and therefore has complexity 1. The first part of \Cref{thm:convex} then follows from \Cref{thm:hypercompabs} and the second part follows from \Cref{cor:hypercomprho}.
\end{proof}

\section{Hypergraph complexity for definable sets}
In this section we prove \cref{prop:o-minimal-complexity}, bounding the complexity of sum hypergraphs associated with a definable set. This implies \Cref{thm:pointvsS} by \Cref{cor:hypercomprho}. As mentioned in the introduction, sets definable with respect to an o-minimal structure have a well-behaved notion of ``dimension''. There are many equivalent ways of defining dimension; we choose one for concreteness.
\begin{defn}[{\cite[Definition 3.14]{Cos00}}]\label{def:dimension}
Say a formula $\phi$ with $d+e$ unbound variables $x_1,\ldots,x_d,y_1,\ldots,y_e$ defines a function $f_\phi:\mathbb{R}^e\to \mathbb{R}^d$ if for every choice of $(x_1,\ldots,x_e)\in \mathbb{R}^e$, there is exactly one choice of $(y_1,\ldots,y_d)\in \mathbb{R}^d$ such that $\phi$ is true, in which case we set $f_\phi(x_1,\ldots,x_e)=(y_1,\ldots,y_d)$. Then for a definable set $S\subset \mathbb{R}^d$, $\dim S$ is the largest $e\in \mathbb{N}$ such that there is an injective function of the form $f_\phi:\mathbb{R}^e\to \mathbb{R}^d$ with range contained in $S$ (in which case we write $f_\phi:\mathbb{R}^e\hookrightarrow S$).
\end{defn}
\begin{rem}
\label{rem:contfphi}
Applying \cite[Proposition 3.17(4)]{Cos00},  \cite[Corollary 3.16]{Cos00}, and \cite[Proposition 2.5]{Cos00} shows that we may take $f_\phi$ to be continuous in the above definition, although we will not need this.
\end{rem}
 We refer the reader to \cite[Section 3.3]{Cos00} for additional basic properties of dimension (such as dimension being well-defined, $\dim \mathbb{R}^d=d$, dimension $0$ sets are exactly the finite subsets of $\mathbb{R}^d$, $\dim(X\cup Y)=\max(\dim X,\dim Y)$, etc.)

The following new notion will be essential to the proof. 
\begin{defn}
We say that a definable set $T$ not containing a line segment is \emph{self-irreducible} if $\dim((T-x)\cap (T-y))<\dim T$ for all distinct $x,y\in \mathbb{R}^d$. We denote by $\mathcal{I}$ the collection of all self-irreducible definable sets.
\end{defn}
Now, \cref{prop:o-minimal-complexity} is an immediate consequence of the following proposition.
\begin{prop}
\label{prop:twofacts}
For any finitely generated o-minimal structure $\RR_{\mathcal F}$, there is a function $f:\NN\to \NN$ such that the following hold.
\begin{compactenum}
    \item Every definable set $S$ with complexity $r$, not containing a line segment, can be written as a union of at most $f(r)$ self-irreducible sets $T\in \mathcal{I}$, each with complexity at most $f(r)$.
    \item For each self-irreducible $T\in \mathcal{I}$ with complexity $r$, and any distinct elements $x,y\in \mathbb{R}^d$, the intersection $(T-x)\cap (T-y)$ is definable, has lower dimension than $T$, and has complexity at most $f(r)$.
\end{compactenum}
\end{prop}
Actually, the second fact in \cref{prop:twofacts} is trivial by the definition of $\mathcal{I}$. Indeed, we always have $\dim((T-x)\cap (T-y))<\dim T$, and for the complexity bound we note that if $\phi_T(z)$ is a formula defining $T$, then $\phi_T(x+z)\wedge \phi_T(y+z)$ is a formula defining $(T-x)\cap (T-y)$.

So, we focus on proving the first of the two facts in \cref{prop:twofacts}. We will need one technical lemma about o-minimal sets not containing a line segment, whose proof we defer until later in this section.
\begin{defn}
Given a definable set $S\subset \mathbb{R}^d$, let $\Bad(S)=\{t\in\mathbb{R}^d\setminus \{\vec0\}: \dim(S\cap (S-t))=\dim S\}$.
\end{defn}
\begin{lem}
\label{lem:BadBounded}
For any finitely generated o-minimal structure $\RR_{\mathcal F}$, there is a function $g:\NN\to \NN$ such that the following holds. Let $S\subset \mathbb{R}^d$ be a definable set with complexity $r$ not containing a line segment. Then $\Bad(S)$ is a finite set, and $|\Bad(S)|\le g(r)$.
\end{lem}

The proof of \cref{lem:BadBounded} requires a few facts about definable sets. For each of these we fix a specific finitely generated o-minimal structure $\RR_{\mathcal F}$. We will use \cite{Cos00} as our reference for o-minimal geometry.
\begin{rem}\label{rem:complexity}
We make a note about our use of ``complexity'' (which doesn't appear in \cite{Cos00}). Given a formula $\phi$ in $a$ unbound variables with $b$ real constants, we can create a formula $\omega$ in $a+b$ unbound variables where all real constants have been replaced with unbound variables. For a given complexity, there are only finitely many such formulas $\omega$, and the definable set $S_{\omega}\subset \mathbb{R}^{a+b}$ has the property that for the projection $\pi:\mathbb{R}^{a+b}\to \mathbb{R}^b$ onto the last $b$ coordinates, the fibers are the definable sets $S_\phi$ for all $\phi$ corresponding to $\omega$. This then gives a ``family of definable sets'', which is the context where the theorems in \cite{Cos00} apply. There is no loss of generality in our formulation because every ``family of definable sets'' is defined by a formula of some complexity, which then bounds the complexity of each fiber of such a family.
\end{rem}
\begin{fact}[The uniform finiteness theorem~{\cite[Theorem 2.9]{Cos00}}]
\label{fact:unifFinite}
For $A\subset \mathbb{R}^d$ a definable set and a linear projection $L:\mathbb{R}^d\to \mathbb{R}^k$, if all fibers of $A\to L(A)$ are finite, then their sizes are bounded by a function of the complexity of $A$.
\begin{cor}
\label{cor:boundorinf}
If $S\subset \mathbb{R}^d$ is a definable set, then either $|S|$ is bounded above by a function of the complexity of $S$, or $S$ is uncountably infinite.
\end{cor}
\begin{proof}
If $S$ were countably infinite, then a generic projection to a 1-dimensional subspace would be a definable countably infinite set, which contradicts that definable subsets of $\mathbb{R}$ are finite unions of points and intervals.

If $|S|$ is finite, then the complexity bound follows by applying \Cref{fact:unifFinite} to the projection to $\mathbb{R}^0$.
\end{proof}
\end{fact}
\begin{cor}
\label{cor:unifFinite}
If $S\subset \mathbb{R}^d$ is a definable set not containing a line segment, then for every line $\ell\subset \mathbb{R}^d$, $|S\cap \ell|$ is bounded by a function of the complexity of $S$.
\end{cor}
\begin{proof}
First, note that $S\cap \ell$ not containing a line segment implies that $S\cap \ell$ is finite, by the defining property of o-minimality after a generic projection to $\mathbb{R}$. Now, consider the definable set
$A=\{(v,w,\lambda)\in \mathbb{R}^d\times \mathbb{R}^d\times\mathbb{R}:w\ne 0\text{ and }v+\lambda w\in S\}$ and the linear projection $L:(v,w,\lambda)\mapsto (v,w)$. Then $A$ has complexity bounded as a function of the complexity of $S$, and the result follows from \Cref{fact:unifFinite}.
\end{proof}
\begin{fact}[Constant-dimension loci are definable~{\cite[Theorem 3.18]{Cos00}}]\label{fact:fiberdimconst}
For a definable set $S\subset \mathbb{R}^d$, a linear projection $L:\mathbb{R}^d\to \mathbb{R}^q$, and any $k\in \NN$, the subset $\{b\in L(S):\dim(L|_S^{-1}(b))=k\}\subseteq L(S)$ is definable and has complexity bounded by the complexity of $S$.
\end{fact}
\begin{fact}
\label{cor:localhomeo}
Suppose $X,Y\subset \mathbb{R}^d$ are definable sets and $\dim X\cap Y=\dim X=\dim Y$. Then there is a Euclidean ball $B(x,r)\subset \mathbb{R}^d$ with $x\in X\cap Y$ such that $X\cap B(x,r)=Y\cap B(x,r)$.
\end{fact}
\begin{proof}
Set $k=\dim(X)=\dim(X\cap Y)=\dim(Y)$, and let $W=X\cup Y \setminus (X\cap Y)$. Then $\dim(\overline{W}\setminus W)\le \dim(W)-1$ by \cite[Theorem 3.22]{Cos00}, so $\dim(\overline{W}\setminus W)\le \dim(X\cup Y)-1= \max(\dim(X),\dim(Y))-1=k-1$. Hence $(X\cap Y)\setminus \overline{W}=(X\cap Y)\setminus (\overline{W}\setminus W)$ is nonempty. Since $\overline{W}$ is closed, any $x\in (X\cap Y)\setminus \overline{W}$ has an open ball $B(x,r)$ not intersecting $W$, meaning $B(x,r)\cap (X\cup Y)\subset X\cap Y$. Therefore $X\cap B(x,r)=(X\cap Y)\cap B(x,r)=Y\cap B(x,r)$ as desired.
\end{proof}

Now, we prove \Cref{prop:twofacts}, assuming \Cref{lem:BadBounded}.

\begin{proof}[Proof of \Cref{prop:twofacts}]
As we have already established the second fact, we only prove the first. We induct on $\dim S$, noting that the theorem is trivial for $\dim S=0$ because then $S$ is a union of a bounded (in terms of $r$) number of points by \Cref{cor:boundorinf}. Recalling \cref{lem:BadBounded}, let $v_1,\dots,v_s$ be an enumeration of $\Bad(S)$, where $s=|\Bad(S)|$ is bounded in terms of the complexity $r$ of $S$. Let $L_i:\mathbb{R}^n\to \mathbb{R}^{n-1}$ be the linear projection onto the orthogonal complement of $v_i$. Then for any $x\in L_i(S)$, the fiber $L_i|_S^{-1}(x)\subset S$ is a definable subset of the line $L_i^{-1}(x)$, and is therefore a finite union of points and intervals. But $S$ does not contain any intervals, so it is a finite union of points.

By the uniform finiteness theorem (\Cref{fact:unifFinite}), there is some $C_r$ depending on $r$ such that each $|L_i^{-1}(x)\cap S|\le C_r$. Then, for $0\le i_1,\dots,i_s\le C_r$, consider the set  $$S_{i_1,\dots,i_s}:=\{x\in S: |(x+\mathbb{R}_{\ge 0}v_j)\cap S|=i_j\text{ for }1\le j \le s\}.$$
That is to say, $S_{i_1,\dots,i_s}$ is the set of points in $S$ for which exactly $i_j$ points of $S$ lie ``ahead'' of $S$ in the $v_j$-direction. Note that this set is definable, because the condition $|(x+\mathbb{R}_{\ge 0}v)\cap S|=i$ may be interpreted as the statement that there exist distinct nonnegative real numbers $\lambda_1=0,\lambda_2,\dots,\lambda_{i}$ such that $x+\lambda_iv\in S$ and there does not exist a non-negative real number $\lambda$ distinct from $\lambda_1,\dots,\lambda_i$ such that $x+\lambda v\in S$.

By the choice of $C_r$, the sets $S_{i_1,\dots,i_s}$ partition $S$. If $\dim S_{i_1,\ldots,i_s}<\dim S$ we can use the inductive hypothesis to decompose $S_{i_1,\ldots,i_s}$ into a bounded (in terms of $r$) number of sets in $\mathcal{I}$. On the other hand, when $\dim S_{i_1,\ldots,i_s}=\dim S$ then we claim that $\Bad(S_{i_1,\dots,i_s})=\emptyset$, meaning that we already have $S_{i_1,\dots,i_s}\in\mathcal I$. To see this, note that $\Bad(S_{i_1,\dots,i_s})\subseteq\Bad(S)=\{v_1,\dots,v_s\}$, but by construction $S_{i_1,\dots,i_s}\cap (S_{i_1,\dots,i_s}+v_j)=\emptyset$ for all $j$.
\end{proof}
It remains to prove \Cref{lem:BadBounded}.
\begin{proof}[Proof of \Cref{lem:BadBounded}]
We first claim that $\Bad(S)$ is definable with complexity bounded in terms of the complexity of $S$. Consider the definable set $A=\{(x,t):t\in \mathbb{R}^d,x\in S\cap (S-t)\}\subset \mathbb{R}^d\times\mathbb{R}^d$, and the projection $\pi:A\to \mathbb{R}^d$ defined by $(x,t)\mapsto t$. We can then write $\Bad(S)=\{t\in \pi(A):\dim \pi^{-1}(t)=\dim(S)\}\subseteq \pi(A)$, so by \Cref{fact:fiberdimconst}, $\Bad(S)$ is definable and has complexity bounded in terms of the complexity of $A$ (which in turn is bounded in terms of the complexity of $S$). By \Cref{cor:boundorinf}, it follows that either $|\Bad(S)|$ is bounded in terms of the complexity of $S$, or is uncountably infinite.

So, it suffices to show that $\Bad(S)$ is countable. Say that $x,y\in S$ are \emph{$r$-isometric} if $(B(x,r)\cap S)+(y-x)=B(y,r)\cap S$, or equivalently $B(x,r)\cap S=B(x,r)\cap (S-(y-x))$. We next claim that for each $t\in \Bad(S)$ there is $x_t\in S$ and $r_t>0$ such that $x_t+t\in S$ and such that $x_t$ and $x_t+t$ are $r_t$-isometric. 

For any $t\in \Bad(S)$ we have  $\dim(S\cap(S-t))=\dim S=\dim(S-t)$, so by \Cref{cor:localhomeo} we obtain $x_t\in S\cap (S-t)$ and $r_t>0$ such that
$S\cap B(x_t,r_t)=(S-t)\cap B(x_t,r_t),$ meaning that $x_t$ and $x_t+t$ are $r_t$-isometric, as desired.

Now, to show $\Bad(S)$ is countable, we show that the set of $t\in \Bad(S)$ with $r_t\ge \varepsilon$ is finite for all $\varepsilon$ (we may then consider a countable sequence $\varepsilon_1,\varepsilon_2,\dots$ converging to zero). Let $F_\varepsilon\subset \mathbb{R}^d\times \mathbb{R}^d$ be the definable set $$F_\varepsilon:=\{(x,t)\in \mathbb{R}^d\times (\mathbb{R}^d\setminus 0): x,x+t\in S\text{ are $\varepsilon$-isometric}\}\subset \mathbb{R}^d\times \mathbb{R}^d.$$
Then $F_\varepsilon$ contains $(x_t,t)$ for every $t\in \Bad(S)$ with $r_t\ge \varepsilon$. Again, let $\pi$ be the projection $(x,t)\mapsto t$, so it suffices to show that the definable set $\pi(F_\varepsilon)$ is finite.

For every $t\in \pi(F_\varepsilon)$, let $x_t\in \mathbb{R}^d$ be some element such that $(x_t,t)\in F_\varepsilon$ (we have already specified such $x_t$ when $t\in \Bad(S)$ and $r_t\ge \epsilon$). Suppose for the purpose of contradiction that $\pi(F_\varepsilon)$ is uncountably infinite. Then for every $\eta<\varepsilon/2$, by the pigeonhole principle there are distinct $t_1,t_2$ such that $\|x_{t_1}-x_{t_2}\|\le \eta$ and $\|t_1-t_2\|\le \eta$.
We claim that
$x_{t_1}$ and $x_{t_1}+(t_1-t_2)$ are $(\varepsilon-2\eta)$-isometric. Indeed, note that for any subsets $D_1\subset B(x_{t_1},\varepsilon)$ and $D_2\subset B(x_{t_2},\varepsilon)$ we have $D_1\cap S=((D_1+t_1)\cap S)-t_1$ and $D_2\cap S=((D_2+t_2)\cap S)-t_2$.
Taking $D_1=B(x_{t_1},\varepsilon-2\eta)$ we have
$$B(x_{t_1},\varepsilon-2\eta)\cap S=(B(x_{t_1}+t_1,\varepsilon-2\eta)\cap S)-t_1,$$
and taking $D_2=B(x_{t_1}+t_1-t_2,\varepsilon-2\eta)\subset B(x_{t_2},\varepsilon)$, we have
$$B(x_{t_1}+t_1-t_2,\varepsilon-2\eta)\cap S=(B(x_{t_1}+t_1,\varepsilon-2\eta)\cap S)-t_2.$$
Adding $t_2-t_1$ to this last equality and comparing with the previous one, we obtain that
$x_{t_1}$ and $x_{t_1}+(t_1-t_2)$ are $(\varepsilon-2\eta)$-isometric as desired.

Iteratively applying this fact, we see that $x_{t_1}+\ell(t_1-t_2)\in S$ for $\ell\in \mathbb{N}$ with $0\le \ell \le (\varepsilon-2\eta)/\eta$. Taking $\eta\to 0$, this gives arbitrarily long arithmetic progressions contained in $S$, contradicting \Cref{cor:unifFinite}.
\end{proof}

\section{Semi-algebraic sets}\label{sec:GL}

In this section we prove \Cref{thm:Gotsman}, giving a bound of the form $\Pr(X\in S)\le (\log n)^{O(1)}/\sqrt n$ when $S$ is semi-algebraic. As discussed in the introduction, the main tool we use in the proof of \Cref{thm:Gotsman} will be the following theorem of Kane~\cite{Kan14} on the total influence (average sensitivity) of polynomial threshold functions.
\begin{thm}[{\cite[Theorem 1.2]{Kan14}}]\label{thm:Kane-GL}
Every degree-$r$ threshold function $f$ has
$$\operatorname{AS}(f)\le \sqrt n (\log n)^{O(r\log r)}2^{O(r^2 \log r)}.$$
\end{thm}
The application of \cref{thm:Kane-GL} will be very similar to an argument due to Kane used to prove bounds for the polynomial Littlewood--Offord problem (see \cite[Section~3]{MNV16}). We will require two lemmas concerning semi-algebraic sets. The first is a bound on the average sensitivity of the indicator function of a semi-algebraic set.

\begin{lem}
\label{lem:semialgsens}
There is a constant $C_r$ so that every semi-algebraic set $S\subset \mathbb{R}^d$ of description complexity $r$ has
$$\operatorname{AS}(1_S)\le C_r\sqrt{n}(\log n)^{C_r}.$$
\end{lem}
\begin{proof}
Given threshold functions $g_1,\dots,g_m$ of degree $r$ and a sign pattern $(\varepsilon_1,\dots,\varepsilon_m)\in\{-1,1\}^m$, it is easy to see that the function $$g(x)=\begin{cases}1 & (g_1(x),\dots,g_m(x))=(\varepsilon_1,\dots,\varepsilon_m)\\0 & \text{otherwise} \end{cases}$$ satisfies $\operatorname{AS}(g)\le \sum_{i=1}^m \operatorname{AS}(g_i)\le m\sqrt n (\log n)^{O(r\log r)}2^{O(r^2 \log r)}$ by \Cref{thm:Kane-GL}. The result follows from applying this to the threshold functions associated to the polynomial conditions exhibiting the degree-complexity $r$ (noting that for a polynomial $f:\mathbb{R}^d\to \mathbb{R}$, the function $1_{f=0}$ is the threshold function for $f^2$).
\end{proof}

The second is the following very weak Littlewood--Offord-type bound, which we will ``boost'' in the course of the proof of \Cref{thm:Gotsman}.

\begin{lem}\label{lem:Algebraic-weak}
There is some $N_{r,d}$ depending only on $r$ and $d$ such that the following holds. Let $S\in \RR^d$ be a semi-algebraic set with description complexity $r$ not containing a line segment. Then if $n\ge N_{r,d}$ we have $\Pr(X\in S)<1$, where $X=a_1\xi_1+\dots+a_n\xi_n$ for independent Rademacher $\xi_1,\dots,\xi_n$.
\end{lem}
\begin{proof}
This follows from \cref{prop:o-minimal-complexity} and \cref{thm:hypercompabs} (we remark that the proof could be greatly simplified in this particular case using the existing machinery of irreducible components for algebraic sets and B\'ezout's theorem).
\end{proof}

\begin{proof}[Proof of \cref{thm:Gotsman}]
Let $N=N_{r,d}$, in the notation of \cref{lem:Algebraic-weak}. We claim that \begin{equation}\Pr(X\in S)\le2^{N} \sum_{i=1}^{N}\Inf_i(1_S).\label{eq:GotsLin}\end{equation}
If we can show this, then by symmetry, the same fact will then hold for the corresponding sum over any set of $N$ indices. Applying this fact to $n/N$ disjoint sets of $N$ indices, it will follow that
$$(n/N)\Pr(X\in S)\le2^{N} \sum_{i=1}^n\Inf_i(1_S)=2^N\operatorname{AS}(1_S),$$
at which point we can conclude with \Cref{lem:semialgsens}.

To prove \cref{eq:GotsLin}, let $\mathcal G$ be the set of all $\xi\in \{0,1\}^n$ such that $X\in S$, and let $\mathcal G'$ be the set of all $\xi\in \{0,1\}^n$ such that the status of the event $X\in S$ changes when we change some $\xi_i$ with $i\le N$. It suffices to prove that $|\mathcal G|\le 2^{N}|\mathcal G'|$. In turn, to prove this it suffices to show that for any $x\in\mathcal G$ there is some $\xi'\in\mathcal G'$ agreeing with $\xi$ on all but the first $N$ coordinates.

To see this, start with some $\xi\in \mathcal G$, meaning that $X(\xi)\in S$. By \cref{lem:Algebraic-weak}, there is some $\xi''\in S$, agreeing with $\xi$ on all but the first $N$ coordinates, such that $X(\xi'')\notin S$. Now, switch from $\xi$ to $\xi''$ by flipping bits one-by-one; along the way we must visit some $\xi'\in \mathcal G'$.
\end{proof}

\section{A strong bound for definable sets}
\label{sec:strongbound}

In this section we prove \cref{thm:baby-forwards}, giving a bound of the form $\Pr(X\in S)\le n^{-1/2+o(1)}$ when $S$ is definable. We start with the following consequence of \cref{thm:pointvsS} and the Tao--Vu inverse theorem~\cite{TV09b}, which gives a strong bound on  $\Pr(X\in S)$ unless most of the coefficients $a_i$ lie in a small generalised arithmetic progression with low rank.

\begin{thm}\label{thm:baby-inverse}
Let $\alpha,C>0$, and $d\in \NN$. Then there are constants $q=q_{d,C,\alpha}\in \NN$ and $B=B_{d,C,\alpha}>0$ such that the following is true.

Consider nonzero $d$-dimensional vectors $a_{1},\dots,a_{n}\in\RR^{d}$, write
$X=a_{1}\xi_{1}+\dots+a_{n}\xi_{n}$, where $\xi_{1},\dots,\xi_{n}$
are i.i.d.\ Rademacher random variables, and let $S\subseteq\RR^d$ be a set which is definable with respect to a finitely generated o-minimal structure $\mathcal F$ and does not contain any line segment. Then for $n$ sufficiently large (in terms of $\mathcal F$ and the complexity of $S$), if $\Pr(X\in S)\ge n^{-C}$, then there is a generalised arithmetic progression $Q$ with rank at most $q$ and size at most $n^{B}$ such that $a_i\in Q$ for all but at most $n^{\alpha}$ indices $i$.
\end{thm}

We will also need the following generalisation of \cref{lem:Algebraic-weak}, in which line segments are allowed, but only in certain directions. In the proof of \cref{thm:baby-forwards}, we will use this to deduce a slight strengthening of \cref{thm:Gotsman}, which will be applied to a semi-algebraic set $\tilde S$ and coefficients $\tilde a_i$ obtained by ``lifting'' $S\subseteq \RR^d$ and $a_1,\dots,a_n\in \RR^d$ into a higher-dimensional space $\RR^q$.

\begin{prop}\label{lem:algebraic-complicated-weak}
Let $W\subseteq \mathbb{R}^q$ be a linear subspace and $\tilde{S}\subseteq \mathbb{R}^q$ a semi-algebraic set of description complexity $r$ such that every line segment in $S$ has direction vector contained in $W$. Then there is $N_{r,q}$ depending only on $r,q$ such that if $n\ge N_{r,q}$, then for any collection of vectors $\tilde a_1,\ldots,\tilde a_n$ with $\tilde a_i\not\in W$, and Rademacher random variables $\xi_1,\ldots,\xi_n$, we have $\Pr(\tilde a_1\xi_1+\dots+\tilde a_n\xi_n\in \tilde S)<1$. 
\end{prop}
\begin{proof}
Instead of description complexity, it is convenient to work with complexity in the o-minimal structure $\RR_\emptyset=\RR_{\mathrm{alg}}$ (by the Tarski--Seidenberg theorem, this is equivalent to working with semi-algebraic description complexity). We show that there is $N_{k,r,q}$ which works for all $\tilde{S}\subset \mathbb{R}^q$ of complexity at most $r$ and dimension at most $k$, by induction on $k$.

For the base case $k=0$, every 0-dimensional semi-algebraic set $\tilde{S}$ is a finite union of points, and the number of points in $\tilde S$ is bounded in terms of $r$ and $q$. So, the existence of suitable $N_{0,r,q}$ follows from the Erd\H os--Littlewood--Offord theorem. Now, consider $\tilde{S}$ with dimension $k\ge 1$.

Let $\pi:\mathbb{R}^q\to \mathbb{R}^d$ be a linear projection with kernel $W$. We first claim that $\pi(\Bad(\tilde{S}))$ is finite. Indeed, suppose it is not. Then, as in the proof of \Cref{lem:BadBounded}, there exists $\varepsilon>0$ and an uncountable collection of $\tilde{t}\in \tilde S$ with $\pi(\tilde{t})$ distinct, for which there are points $\tilde{x}_{\tilde{t}}\in \tilde{S}$ such that $\tilde{x}_{\tilde{t}}$ and $\tilde{x}_{\tilde{t}}+\tilde{t}$ are $\varepsilon$-isometric. Then, again as in the proof of \Cref{lem:BadBounded}, for any $\eta>0$ we can find an arithmetic progression of length $(\varepsilon-2\eta)/\eta$ contained in $\tilde{S}$, in some direction $\tilde t_1-\tilde t_2$. Since $\pi(\tilde t_1)\ne \pi(\tilde t_2)$, this direction is not contained in $W$. Taking $\eta>0$ arbitrarily small, we have found arbitrarily many points of $\tilde S$ contained in a line $\ell$ not parallel to $W$. But this is a contradiction, as the number of such points must be bounded as a function of $r$ and $q$, by the same proof as for \Cref{cor:unifFinite} (with $w$ constrained not to lie in $W$).

Let $M=M_{r,q}$ be large enough (in terms of $r$ and $q$) such that $|\pi(\Bad(\tilde{S}))|\le M$ (the existence of such an $M$ follows from \Cref{cor:boundorinf} as  $\pi(\Bad(\tilde{S}))$ is finite and definable with complexity bounded in terms of $r$). Write $X_M=\tilde a_1\xi_1+\dots+\tilde a_M\xi_M$ and $X'=\tilde a_{M+1}\xi_{M+1}+\dots \tilde a_n \xi_n$, so $X=X_M+X'$. Now, given vectors $\tilde w_0,\ldots,\tilde w_{M+1}$ with $\pi(\tilde w_i)$ distinct, we then have $$\dim\left(\bigcap_{i=0}^{M+1} (\tilde {S}-\tilde w_i)\right)<\dim \tilde{S}\le k.$$ Indeed, by the pigeonhole principle, for some $i$ we have $\tilde{w}_i-\tilde{w}_0\not\in \Bad(\tilde{S})$, meaning that $\bigcap_{i=0}^M (\tilde{S}-\tilde w_i)\subseteq -\tilde{w}_0+\left(\tilde{S}\cap (\tilde{S}-(\tilde{w}_i-\tilde{w}_0))\right)$ has dimension less than $\dim \tilde{S}$. But the random variable $\pi(X_M)=\pi(\tilde a_1)\xi_1+\dots+\pi(\tilde a_M)\xi_M$ attains at least $M+1$ different values. So, for $X=\tilde a_1\xi_1+\dots \tilde a_n \xi_n$ to be supported in $S$, it must be the case that $X'$ is supported in 
$$\tilde{S}':= \bigcap_{\tilde w\in \operatorname{supp}(X_M)}(\tilde{S}-\tilde w),$$
which has dimension at most $k-1$.
The complexity of $\tilde{S}'$ is bounded above by a function $r'$ of $r$ and $q$, so for this $r'$ we can choose $N_{k,r,q}=M_{r,q}+ N_{k-1,r',q}$.
\end{proof}
The final ingredient we will need for the proof of \cref{thm:baby-forwards} is the following powerful theorem of Pila. In this statement, ``open'' is with respect to the Euclidean topology on $\RR^q$.
\begin{thm}[{\cite[Theorem 3.5(2)]{P09}}]\label{thm:pila}
Fix $\beta>0$, $q\in \NN$ and a finitely generated o-minimal structure $\RR_{\mathcal F}$, and suppose $N$ is sufficiently large in terms of $\beta,q,\mathcal F$. For any definable set $\tilde{S}\subseteq \mathbb{R}^q$ of complexity $r$, there are $s\le N^\beta$ open subsets $\tilde U_1,\dots,\tilde U_s$ of semi-algebraic sets $\tilde T_1,\dots, \tilde T_s$, with each $\tilde U_i\subseteq \tilde S$ and each $\tilde T_i$ with description complexity bounded in terms of $\mathcal F,r,q$, such that $\{0,\ldots,N\}^q \cap \tilde{S}\subseteq \tilde U_1\cup\dots\cup \tilde U_s$.
\end{thm}
\begin{rem}
The statement of \cref{thm:pila} above is not exactly the same as \cite[Theorem 3.5(2)]{P09}. Indeed, Pila states his theorem in terms of a ``basic block family'', and technically this statement only implies that the semi-algebraic sets $\tilde T_1,\dots,\tilde T_s$ have complexity bounded in terms of $r,q$, when interpreted as sets definable with respect to $\RR_{\mathcal F}$. This is not quite enough to guarantee a bound on description complexity (such a bound would follow from the Tarski--Seidenberg theorem only if ${\mathcal F}=\emptyset$). However, in the construction in the proof of \cite[Theorem 3.5(2)]{P09}, one can see directly that $\tilde T_1,\dots,\tilde T_s$ have bounded description complexity.
\end{rem}
Now we prove \cref{thm:baby-forwards}.
\begin{proof}[Proof of \cref{thm:baby-forwards}]
Let $\mathcal F$ be the generators for the o-minimal structure that $S$ is defined with respect to. Applying \cref{thm:baby-inverse} with $C=1/2$ and say $\alpha=1/2$, we immediately have $\Pr(X\in S)\le 1/\sqrt n$ (for large $n$) unless there is a generalised arithmetic progression $Q$ with rank at most $q=O(1)$ and size at most $n^B=n^{O(1)}$, such that $a_i\in Q$ for at least $n/2$ indices $i$. So, we may assume there is such a $Q$, and without loss of generality we may assume $a_1,\dots,a_{\floor{n/2}}\in Q$. Let $X_1=\sum_{i\le \lfloor n/2 \rfloor} a_i\xi_i$ and $X_2=\sum_{i>\lfloor n/2 \rfloor} a_i\xi_i$, and condition on any outcome of $X_2$ and let $S'=S-X_2$. For the remainder of the proof our goal is to show that $\Pr(X_1\in S')\le n^{-1/2+o(1)}$. Let $m=\lfloor n/2\rfloor$ in what follows.

We may assume that $Q$ is homogeneous (having base point $b=0$), by adding an additional generator if necessary. Consider the projection $\pi:\mathbb{R}^q\to \mathbb{R}^d$ sending the standard basis vectors in $\RR^q$ to the generators of $Q$, and let $\tilde a_1,\ldots,\tilde a_m\in \{0,\ldots,n^B\}^q$ be such that $\pi(\tilde a_i)=a_i$ for each $i$. Then with $\tilde{X_1}$ the random variable $\tilde{a}_1\xi_1+\dots+\tilde{a}_m\xi_m$, and $\tilde S'$ the definable set $\pi^{-1}(S')$, we have $$\Pr(X_1\in S')= \Pr(\tilde{X}_1\in \tilde S')=\Pr(\tilde{X}_1\in \tilde S'\cap \{0,\ldots,n^{B+1}\}^q).$$

Now, fix $\varepsilon>0$. From now on, the implicit constants in asymptotic notation are allowed to depend on $q,r,\mathcal F,\varepsilon$. By \cref{thm:pila} (with $N=n^{B+1}$ and $\beta=\varepsilon/(2B+2)$), for some $s\le n^{\varepsilon/2}$ there are open subsets $\tilde U_1,\dots,\tilde U_s$ of semi-algebraic sets $\tilde T_1,\dots,\tilde T_s$, with each $\tilde U_i\subseteq \tilde S'$ and each $\tilde T_i$ with description complexity $O(1)$, such that $\tilde S'\cap \{0,\ldots,n^{B+1}\}^q\subseteq \tilde U_1\cup \dots\cup \tilde U_s$.

Since each $\pi(\tilde U_i)\subseteq S$ and $S$ contains no line segment, any line segment that might be contained in $\tilde U_i$ must have direction vector in the kernel $W$ of $\pi$. Now, for each $i$ let $\tilde T_i'$ be the subset of $\tilde T_i$ obtained by removing all line segments with direction vectors not in $W$. Note that $\tilde T_i'$ is itself a semi-algebraic set with complexity $O(1)$: indeed, $\tilde T_i'$ can be described by a first-order formula with length $O(1)$, so this follows from the Tarski--Seidenberg theorem. If an open subset of $\tilde T_i$ contains a single element of a line segment in $\tilde T_i$, then it contains an entire sub-segment. So, $\tilde U_i\subseteq \tilde T_i'$. Now, using \cref{lem:algebraic-complicated-weak} instead of \cref{lem:Algebraic-weak}, the proof of \cref{thm:Gotsman} shows that $$\Pr(\tilde{X}_1\in \tilde U_i)\le \Pr(\tilde{X}_1\in \tilde T_i')\le m^{-1/2+o(1)}=n^{-1/2+o(1)}$$ for each $i$. Taking the union bound over all $i$, and assuming $n$ is sufficiently large, we have \[\Pr(X_1\in S')=\Pr(\tilde{X}_1\in \tilde S'\cap \{0,\ldots,n^{B+1}\}^q)\le \sum_{i=1}^s\Pr(\tilde{X}_1\in \tilde U_i)\le sn^{-1/2+o(1)}\le n^{-1/2+\varepsilon}.\]
We conclude by taking $\varepsilon\to 0$.
\end{proof}

\section{Concluding remarks}\label{sec:concluding}
In this paper we have proposed some natural geometric variants of the Littlewood--Offord problem, and proved some bounds in several cases. Of course, \cref{conj:forwards} remains open for many natural choices of $S$. The simplest open case is where $S\in \RR^4$ is the three-dimensional unit sphere in four-dimensional space. It would also be of interest to improve the bounds in \cref{thm:convex} for sets of points in convex position: can we at least get bounds of the form $\Pr(X\in S)\le n^{-1/2+o(1)}$ in this case?

\begin{conjecture}
Let $S\subseteq\RR^d$ be a set of points in convex position, and let $\varepsilon>0$. If $n$ is sufficiently large in terms of $\varepsilon,d$, then the following holds.
Consider nonzero $d$-dimensional vectors $a_{1},\dots,a_{n}\in\RR^{d}$, and write
$X=a_{1}\xi_{1}+\dots+a_{n}\xi_{n}$, where $\xi_{1},\dots,\xi_{n}$
are i.i.d.\ Rademacher random variables. Then $$\Pr(X\in S)\le n^{-1/2+\varepsilon}.$$
\end{conjecture}
The analogue of \cref{conj:forwards} also plausibly holds in this case.

Regarding \cref{thm:pointvsS}, it would be good to understand the best possible bound on $\Pr(X\in S)$ in terms of the point probability $\rho=\max_{x\in \RR^d}\Pr(X=x)$. If one were able to prove that $\Pr(X\in S)=O(\rho^{1/k})$, this would actually imply \cref{conj:forwards}, using Hal\'asz' theorem (\cref{thm:Halasz}) in a similar way to the proof of \cref{thm:variousS}.

Also, recall that the proof of \cref{thm:Gotsman} proceeded via a bound on the average sensitivity of polynomial threshold functions, related to the Gotsman--Linial conjecture. As we discussed in \cref{subsec:polynomial}, a full resolution of the Gotsman--Linial conjecture would imply \cref{conj:forwards} for affine varieties, but actually the following ``bounded'' version of the Gotsman--Linial conjecture would suffice, and may be of independent interest.

\begin{conjecture}\label{conj:GL-case}
Let $S\subseteq \RR^d$ be a semi-algebraic set with description complexity $r$, and for nonzero coefficients $a_1,\dots,a_n\in \RR^n$ let $f(\xi_1,\dots,\xi_n)$ be the Boolean function measuring whether $a_1\xi_1+\dots+a_n\xi_n\in S$. Then $f$ has average sensitivity at most $C_{r,d}\sqrt n$, for some $C_{r,d}$ depending only on $r$ and $d$.
\end{conjecture}

\cref{conj:GL-case} is equivalent to the special case of the Gotsman--Linial conjecture where we only consider polynomials of the form $f\circ \pi$, where $f$ is a $d$-variable polynomial and $\pi:\RR^n\to \RR^d$ is a linear projection.

We remark that in the special case of dimension $d=2$, even though we were able to prove \cref{conj:forwards}, it is not clear how to prove \cref{conj:GL-case}. It might also be interesting to consider the case where $S$ is a convex or definable set.

It would also be interesting to better understand which sets have the generic intersection property in \cref{def:generic}. In particular, it would be nice to know whether a random algebraic hypersurface in $\mathbb{R}^d$ (conditioned on not containing a line) has the generic intersection property almost surely. This would imply that \cref{conj:forwards} holds for ``almost all algebraic hypersurfaces''.

\textbf{Acknowledgements:} We would like to thank Jonathan Pila for a number of insightful comments and suggestions.

\bibliographystyle{amsplain_initials_nobysame_nomr}
\bibliography{bib}

\end{document}